\def\year{2021}\relax
\documentclass[letterpaper]{article} % DO NOT CHANGE THIS
\usepackage{aaai21}  % DO NOT CHANGE THIS
\usepackage{times}  % DO NOT CHANGE THIS
\usepackage{helvet} % DO NOT CHANGE THIS
\usepackage{courier}  % DO NOT CHANGE THIS
\usepackage[hyphens]{url}  % DO NOT CHANGE THIS
\usepackage{graphicx} % DO NOT CHANGE THIS
\usepackage{graphicx}  % DO NOT CHANGE THIS
\usepackage{natbib}  % DO NOT CHANGE THIS AND DO NOT ADD ANY OPTIONS TO IT
\usepackage{caption} % DO NOT CHANGE THIS AND DO NOT ADD ANY OPTIONS TO IT
\title{Robust Reinforcement Learning: A Case Study in Linear Quadratic Regulation}
\author{

    Bo Pang, Zhong-Ping Jiang\\
}
\affiliations{
    Department of Electrical and Computer Engineering, New York University, Six Metrotech Center, Brooklyn, NY 11201\\
    \{bo.pang, zjiang\}@nyu.edu \\
}
\usepackage{physics,esdiff,mathtools}
\usepackage{amsmath,amssymb,amsfonts,amsthm,mathrsfs,comment,multirow}
\usepackage[inline]{enumitem}
\usepackage[ruled,linesnumbered]{algorithm2e}
\SetKw{Break}{break}
\newtheorem{theorem}{Theorem}
\newtheorem{lemma}{Lemma}

\newtheorem{definition}{Definition}
\newtheorem{assumption}{Assumption}
\newtheorem{problem}{Problem}

\newtheorem{proc}{Procedure}

\DeclareMathOperator{\vect}{vec}
\DeclareMathOperator{\svec}{svec}
\DeclareMathOperator{\blkdiag}{blkdiag}

\begin{document}

\maketitle

\begin{abstract}
This paper studies the robustness of reinforcement learning algorithms to errors in the learning process. Specifically, we revisit the benchmark problem of discrete-time linear quadratic regulation (LQR) and study the long-standing open  question: Under what conditions is the policy iteration method robustly stable from a dynamical systems perspective? Using advanced stability results in control theory, it is shown that policy iteration for LQR is inherently robust to small errors in the learning process and enjoys small-disturbance input-to-state stability: whenever the error in each iteration is bounded and small, the solutions of the policy iteration algorithm are also bounded, and, moreover, enter and stay in a small neighbourhood of the optimal LQR solution. As an application, a novel off-policy optimistic least-squares policy iteration for the LQR problem is proposed, when the system dynamics are subjected to additive stochastic disturbances. The proposed new results in robust reinforcement learning are validated by a numerical example.
\end{abstract}

%%%%%%%%%%%%%%%%%%%%%%%%%%%%%%%%%%%%%%%%%%%%%%%%%%%%%%%%%%%%%%%%%%%%%%%%%%%%%%%%
\section{Introduction}
As an important and popular method in reinforcement learning (RL), policy iteration has been widely studied by researchers and utilized in different kinds of real-life applications by practitioners \cite{BertsekasOptimalControl2,sutton2018reinforcement}. Policy iteration involves two steps, \textit{policy evaluation} and \textit{policy improvement}. In policy evaluation, a given policy is evaluated based on a scalar performance index. Then this performance index is utilized to generate a new control policy in policy improvement. These two steps are iterated in turn, to find the solution of the RL problem at hand. When all the information involved in this process is exactly known, the convergence to the optimal solution can be provably guaranteed, by exploiting the monotonicity property of the policy improvement step. That is, the performance of the newly generated policy is no worse than that of the given policy in each iteration. Over the past decades, various versions of policy iteration have been proposed, for diverse optimal control problems, see \cite{BertsekasOptimalControl2,sutton2018reinforcement,LewisBookOptimalControl,jiang2017robust,8c288426ea8c40d2a53f0731c99311a4} and the references therein.

In reality, policy evaluation or policy improvement can hardly be implemented precisely, because of the existence of various errors, which may be induced by function approximation, state estimation, sensor noise, external disturbance and so on. Therefore, a natural question to ask is: when is a policy iteration algorithm robust to errors in the learning process? In other words, under what conditions on the errors, does the policy iteration still converge to (a neighbourhood of) the optimal solution? And how to quantify the size of this neighbourhood? In spite of the popularity and empirical successes of policy iteration, its robustness issue has not been fully understood yet in theory, due to the inherent nonlinearity of the process \cite{bertsekas2011approximate}. The problem becomes more complex when the state and action spaces are unbounded and continuous, which are common in RL problems of physical systems such as robotics and autonomous cars \cite{DBLPjournals}. Indeed, in this case the stability issue needs to be addressed, to avoid the selection of destabilizing policies that drive the states of the closed-loop system into the infinity or an unsafe region.

In this paper, we investigate the robustness of policy iteration for the discrete-time linear quadratic regulator (LQR) problem, which was firstly proposed in \cite{1099755}. Even if the LQR is the most basic and important optimal control problem with unbounded, continuous state and action spaces \cite{BertsekasOptimalControl2}, the robustness of its associated policy iteration to errors in the learning process has not been fully investigated. The main idea of this paper is to regard the policy iteration as a dynamical system, and then utilize the concepts of exponential stability and input-to-state stability in control theory to analyze its robustness \cite{sontag2008input}. To be more specific, we firstly prove that the optimal LQR solution is a locally exponentially stable equilibrium of the exact policy iteration (see Lemma \ref{lemma_local_exponentail_stable}). Then based on this observation, we show that the policy iteration with errors is locally input-to-state stable, if the errors are regarded as the disturbance input (see Lemma \ref{lemma_local_ISS}). That is, if the policy iteration starts from an initial solution close to the optimal solution, and the errors are small and bounded, the discrepancies between the solutions generated by the policy iteration and the optimal solution will also be small and bounded. Thirdly, we demonstrate that for any initial stabilizing control gain, as long as the errors are small, the approximate solution given by policy iteration will eventually enter a small neighbourhood of the optimal solution (see Theorem \ref{theorem_global}). Finally, a novel off-policy model-free RL algorithm, named optimistic least-squares policy iteration (O-LSPI), is proposed for the LQR problem with dynamics perturbed by additive stochastic disturbances. Our robustness result is applied to show the convergence of this off-policy O-LSPI (see Theorem \ref{theorem_stochastic_algorithm_convergence}). Experiments on a numerical example validate our results.

Our main contributions are two-fold. First, we provide a control-theoretic robustness analysis for the policy iteration of discrete-time LQR. Second, we propose a novel off-policy RL algorithm O-LSPI with provable convergence.

In the rest of this paper, we first present some preliminaries, followed by the robustness analysis and the off-policy O-LSPI. Then we present the experimental results, discuss some related work, and close the paper with some concluding remarks.

\subsection{Notations} 
$\mathbb{R}$ ($\mathbb{R}_+$) is the set of all real (nonnegative) numbers; $\mathbb{Z}_+$ denotes the set of nonnegative integers; $\mathbb{S}^{n}$ is the set of all real symmetric matrices of order $n$; $\otimes$ denotes the Kronecker product; $I_n$ denotes the identity matrix with dimension $n$; $\Vert \cdot\Vert_F$ is the Frobenius norm; $\Vert\cdot\Vert_2$ is the 2-norm for vectors and the induced 2-norm for matrices; for signal $Z:\mathbb{F}\rightarrow \mathbb{R}^{n\times m}$, $\Vert Z\Vert_\infty$ denotes its $l^\infty$-norm when $\mathbb{F} = \mathbb{Z}_+$, and $L^\infty$-norm when $\mathbb{F} = \mathbb{R}_+$. For matrices $X\in \mathbb{R}^{m\times n}$, $Y\in \mathbb{S}^m$, and vector $v\in \mathbb{R}^{n}$, define 
\begin{align*}
    \vect(X) &= [\begin{array}{cccc}
    X_1^T & X_2^T & \cdots & X_n^T
    \end{array}]^T,\ \tilde{v} = \svec(vv^T),\\
    \svec(Y) &= [y_{11},\sqrt{2}y_{12},\cdots,\sqrt{2}y_{1m},y_{22},\sqrt{2}y_{23}, \\ 
    &\cdots,\sqrt{2}y_{m-1,m},y_{m,m}]^T\in \mathbb{R}^{\frac{1}{2}m(m+1)},
\end{align*}
where $X_i$ is the $i$th column of $X$. For $Z\in\mathbb{R}^{m\times n}$, define $\mathcal{B}_{r}(Z) = \{X\in\mathbb{R}^{m\times n}\vert \Vert X - Z\Vert_F<r\}$ and $\mathcal{\bar{B}}_{r}(Z)$ as the closure of $\mathcal{B}_{r}(Z)$. $Z^\dagger$ is the Moore-Penrose inverse of matrix $Z$. $\blkdiag(Z_1,Z_2,\cdots,Z_N)$ refers to
the block-diagonal matrix that consists of a set of matrices $Z_1,Z_2,\cdots,Z_N$.

\section{Preliminaries}\label{section_problem_formulation}
Consider linear time-invariant systems of the form
\begin{equation}\label{LTI_sys}
    x_{k+1} = Ax_k + Bu_k,\quad x_0 = x_{ini}
\end{equation}
where $x_k\in\mathbb{R}^n$ is the system state, $u_k\in\mathbb{R}^m$ is the control input, $x_{ini}\in\mathbb{R}^n$ is the initial condition, $A\in\mathbb{R}^{n\times n}$ and $B\in\mathbb{R}^{n\times m}$. $\left(A,B\right)$ is controllable, that is, $[B, AB, A^2B, \cdots, A^{n-1}B]$ has full row rank. The classic LQR problem is to find a controller $u$ in order to minimize the following cost functional
\begin{equation}\label{cost_deterinistic}
    J(x_0,u) = \sum_{k=0}^\infty c(x_k,u_k),
\end{equation}
where $c(x_k,u_k) = x_k^TSx_k + u_k^TRu_k$, $S\in\mathbb{S}^n$ is positive semidefinite and $R\in\mathbb{S}^n$ is positive definite. $(A,S^{1/2})$ is observable, that is, $(A^T,S^{1/2})$ is controllable. It is well-known that in such a setting, the LQR problem admits a unique optimal controller $u^* = -K^*x$, where
\begin{equation}\label{optimal_gain}
    K^*=(R+B^TP^*B)^{-1}B^TP^*A
\end{equation}
with $P^*\in\mathbb{S}^n$ the unique positive definite solution of the algebraic Riccati equation (ARE)
\begin{equation}\label{ARE}
A^TPA - P - A^TPB(R + B^TPB)^{-1}B^TPA + S = 0.
\end{equation}
In addition, $A-BK^*$ is stable, i.e., the spectral radius $\rho(A-BK^*)<1$. See \cite[Section 2.4]{LewisBookOptimalControl} for details. For convenience, a control gain $K\in\mathbb{R}^{m\times n}$ is said to be stabilizing if $A-BK$ is stable.

\subsection{Policy Iteration for LQR}
For any stabilizing control gain $K\in\mathbb{R}^{m\times n}$, the cost \eqref{cost_deterinistic} with $u_k = -Kx_k$ is a quadratic function of the initial state \cite[Section 2.4]{LewisBookOptimalControl}. Specifically, $J(x_0,-Kx) = x_0^TP_Kx_0$, where $P_K\in\mathbb{S}^n$ is the unique positive definite solution of the Lyapunov equation
\begin{equation}\label{algebraic_Lyapunov_equation}
    (A-BK)^TP_K(A-BK) - P_K + S + K^TRK = 0.
\end{equation}
Define function
\begin{equation*}
\begin{split}
    G(P_K) &= \left[
                \begin{array}{c|c}
                [G(P_K)]_{xx} & [G(P_K)]_{ux}^T \\
                \hline
                [G(P_K)]_{ux} & [G(P_K)]_{uu}
        \end{array}
        \right] \\ 
        &\triangleq \left[\begin{array}{cc}
                S + A^TP_KA - P_K & A^TP_KB \\
                B^TP_KA & R + B^TP_KB
        \end{array}\right].            
\end{split}
\end{equation*}
Then \eqref{algebraic_Lyapunov_equation} can be rewritten as
$$\mathcal{H}(G(P_K),K) = 0,$$
where
$$\mathcal{H}(G(P_K),K) \triangleq \left[\begin{array}{cc}
                I_n &  -K^T
                \end{array}\right] G(P_K)\left[\begin{array}{c}
                I_n  \\
                -K
                \end{array}\right].$$
The policy iteration for LQR is presented below, which is an equivalent reformulation of the original results in \cite{1099755}.
\begin{proc}[Exact Policy Iteration]\label{procedure_policy_itration}
\ \par
\begin{enumerate}[label=\arabic*),leftmargin = 0.5cm]
    \item Choose a stabilizing control gain $K_1$, and let $i=1$.
    \item\label{standard_PI_PE} (Policy evaluation) Evaluate the performance of control gain $K_i$, by solving 
    \begin{equation}\label{RPI_PE}
       \mathcal{H}(G_i,K_i)=0
    \end{equation}
    for $P_i\in\mathbb{S}^n$, where $G_i\triangleq G(P_i)$.
    \item (Policy improvement) \label{standard_PI_PI} Obtain an improved policy
    \begin{equation}\label{RPI_PI}
        K_{i+1} = [G_i]^{-1}_{uu}[G_i]_{ux}.
    \end{equation}
    \item Set $i\gets i+1$ and go back to Step \ref{standard_PI_PE}.
\end{enumerate}
\end{proc}
The following convergence results of Procedure \ref{procedure_policy_itration} were also provided in \cite{1099755}.
\begin{theorem}\label{theorem_standard_PI}
    In Procedure \ref{procedure_policy_itration} we have:
    \begin{enumerate}[label=\roman*), leftmargin=0.6cm]
        \item $A-BK_i$ is stable for all $i=1,2,\cdots$.  \label{Hurwitz_matrix}
        \item $P_1\geq P_2 \geq P_3\geq \cdots \geq P^*$.\label{PI_monotone}
        \item $\lim_{i\rightarrow\infty}P_i=P^*$, $\lim_{i\rightarrow\infty}K_i = K^*$.\label{PI_converge}
    \end{enumerate}
\end{theorem}
%\begin{proof}
%Equation \eqref{RPI_PE} along with \eqref{RPI_PI} is equivalent to \cite[equation (11)]{1098829}. Hence Theorem \ref{theorem_standard_PI} holds by \cite[Theorem]{1098829}.
%\end{proof}
\subsection{Problem Formulation}
In Procedure \ref{procedure_policy_itration}, the exact knowledge of $A$ and $B$ is required, as the solution to \eqref{RPI_PE} relies upon $A$ and $B$. So, the exact policy iteration is model-based. However, in practice, very often we only have access to incomplete information required to solve the problem. In other words, each policy evaluation step will result in inaccurate estimation. Thus we are interested in studying the following problem.
\begin{problem}\label{problem_robustness}
If $G_i$ is replaced by an approximated matrix $\hat{G}_i$, will the conclusions in Theorem \ref{theorem_standard_PI} still hold?
\end{problem}
The difference between $\hat{G}_i$ and $G_i$ can be attributed to errors from various sources. One example comes from the problem of using reinforcement learning method to find the optimal solutions for LQR when \eqref{LTI_sys} is subjected to additive external disturbances. Concretely, consider system \eqref{LTI_sys} perturbed by external noise
\begin{equation}\label{LTI_sys_stochastic}
    x_{k+1} = Ax_k + Bu_k + Cw_k, \quad x_0 = x_{ini}
\end{equation}
where $C\in\mathbb{R}^{n\times q}$, $w_k\in\mathbb{R}^q$ is drawn i.i.d. from the standard Gaussian distribution $\mathcal{N}(0,I_q)$, and matrices $A$, $B$ and $C$ are unknown. Since the information about system matrices is unavailable, we need to implement the policy evaluation using input/state data. Due to the existence of unmeasurable stochastic noise $w_k$, generally we could only obtain an estimation $\hat{G}_i$ of the true $G_i$ from the noise-corrupted input/state data. Other sources that cause the difference between $\hat{G}_i$ and $G_i$ include but are not limited to: the estimation errors of $A$ and $B$ in indirect adaptive control, system identification and model-based reinforcement learning \cite{astrom1995adaptive,Ljung1999identification,pmlr-v99-tu19a}; the residual caused by an early termination of the iteration to numerically solve ARE (\ref{ARE}), in order to save computational efforts \cite{hylla2011extension}; approximate values of $S$ and $R$ in inverse optimal control/imitation learning, due to the absence of exact knowledge of the cost function \cite{Levine2012inverseOptimalControl,Monfort2015inverseLQR}.

In this work, using the concept of exponential stability and input-to-state stability in control theory, we provide an answer to Problem \ref{problem_robustness}. Moreover, we provide the convergence analysis of the novel O-LSPI when it is applied to solve the LQR problem for uncertain systems \eqref{LTI_sys_stochastic}.

\subsection{Notions of Exponential and Input-to-State Stability}
Consider a dynamical system of the general form
\begin{equation}\label{nonlinear_sys}
    x_{k+1} = f(x_k,u_k),\quad x_0 = x_{ini},
\end{equation}
where $x_k\in\mathbb{R}^n$, $u_k\in\mathbb{R}^m$, $f:\mathbb{R}^n\times\mathbb{R}^m\rightarrow\mathbb{R}^n$ is continuous, and $x^*$ is an equilibrium of $x_{k+1} = f(x_k,0)$ when $u_k=0$ for all $k\in\mathbb{Z}_+$. The concepts of exponential and input-to-state stability for \eqref{nonlinear_sys} are recalled in this subsection. See \cite{JIANG20042129} for more details. \begin{definition}\label{definition_exp_stab_nonlinear}
For \eqref{nonlinear_sys} with $u_k=0$ for all $k\in\mathbb{Z}_+$, $x^*$ is a locally exponentially stable equilibrium if there exists a $\delta>0$, such that for some $a>0$ and $0<b<1$,
$$\Vert x_k - x^*\Vert_2\leq ab^k\Vert x_{ini}-x^*\Vert_2$$
for all $x_{ini}\in\mathcal{B}_{\delta}(x^*)$. If $\delta=+\infty$, then $x^*$ is a globally exponentially stable equilibrium.
\end{definition}
The exponential stability implies not only the convergence, but also the convergence rate of \eqref{nonlinear_sys}. When the input signal is not zero, the input-to-state stability characterizes how the solution of \eqref{nonlinear_sys} is affected by the input signal.
\begin{definition}
A function $\gamma:\mathbb{R}_+\rightarrow\mathbb{R}_+$ is said to be of class $\mathcal{K}$ if it is continuous, strictly increasing and vanishes at the origin. A function $\beta:\mathbb{R}_+\times \mathbb{R}_+\rightarrow\mathbb{R}_+$ is said to be of class $\mathcal{K}\mathcal{L}$ if $\beta(\cdot,t)$ is of class $\mathcal{K}$ for every fixed $t\in\mathbb{R}_+$ and, for every fixed $r\geq 0$, $\beta(r,t)$ decreases to $0$ as $t\rightarrow\infty$.
\end{definition}
\begin{definition}\label{Definition_local_ISS}
System \eqref{nonlinear_sys} is locally input-to-state stable if there exist some $\alpha_1>0$, some $\alpha_2>0$, some $\beta\in\mathcal{K}\mathcal{L}$ and some $\gamma\in\mathcal{K}$, such that for each $u$ and each $x_{ini}$ satisfying $x_{ini}\in\mathcal{B}_{\alpha_1}(x^*)$, $\Vert u\Vert_\infty<\alpha_2$, the corresponding solution $x_k$ satisfies
$$\Vert x_k-x^*\Vert_2\leq \beta(\Vert x_{ini}-x^*\Vert_2,k) + \gamma(\Vert u\Vert_\infty).$$
\end{definition}
Literally speaking, the local input-to-state stability implies that the distance from the state to the equilibrium is bounded if the input signal is small and the initial state is close to the equilibrium. In addition, the effect of the initial condition vanishes as time goes to infinity.

\section{Robustness Analysis of Policy Iteration}\label{section_robust_analysis}
Consider the policy iteration in the presence of errors.
\begin{proc}[Inexact Policy Iteration]\label{procedure_robust_policy_iteration}
\ \par
 \begin{enumerate}[label=\arabic*),leftmargin=0.5cm]
        \item Choose a stabilizing control gain $\hat{K}_1$, and let $i=1$.
        \item\label{inexact_PI_PE} (Inexact policy evaluation) Obtain $\hat{G}_i = \tilde{G}_i + \Delta G_i$, where $\Delta G_i\in\mathbb{S}^{m+n}$ is a disturbance,
        $\tilde{G}_i\triangleq G(\tilde{P}_i)$
        and $\tilde{P}_i\in\mathbb{S}^n$ satisfy
        \begin{equation}\label{eRPI_PE}
        \mathcal{H}(\tilde{G}_i,\hat{K}_i)=0,
        \end{equation}
        and $J(x_0,-\hat{K}_ix) = x_0^T\tilde{P}_ix_0$ is the true cost induced by control gain $\hat{K}_i$, if $\hat{K}_i$ is stabilizing.
        \item (Policy update) Construct a new control gain
        \begin{equation}\label{eRPI_PI}
            \hat{K}_{i+1} = [\hat{G}_i]_{uu}^{-1}[\hat{G}_i]_{ux}.
        \end{equation}
        \item Set $i\gets i+1$ and go back to Step \ref{inexact_PI_PE}.
    \end{enumerate}
\end{proc}
% The requirement that $\hat{G}_i\in\mathbb{S}^{m+n}$ in Procedure \ref{procedure_robust_policy_iteration} is not restrictive, since for any $X\in\mathbb{R}^{(n+m)\times(n+m)}$, $x^TXx = \frac{1}{2}x^T(X+X^T)x$, where $\frac{1}{2}(X+X^T)$ is symmetric.

We firstly show that the exact policy iteration Procedure \ref{procedure_policy_itration}, viewed as a dynamical system, is locally exponentially stable at $P^*$. Then based on this result, we show that the inexact policy iteration, viewed as a dynamical system with $\Delta G_i$ as the input, is locally input-to-state stable.

For $X\in\mathbb{R}^{n\times n}$, $Y\in\mathbb{R}^{n\times n}$, define
\begin{equation*}
    \begin{split}
        \mathscr{A}(X) &= X^T\otimes X^T- I_n\otimes I_n,\ \mathcal{L}_{X}(Y)=X^TYX - Y,\\  
        \mathscr{K}(Y) &= \mathscr{R}^{-1}(Y)B^TYA,\\
        \mathscr{R}(Y) &= R + B^TYB,\quad \mathcal{A}(\mathscr{K}(Y)) = A - B\mathscr{K}(Y).
    \end{split}
\end{equation*}
Then obviously
\begin{equation}\label{equivalent_operator_matrix}
    \vect(\mathcal{L}_{X}(Y)) = \mathscr{A}(X)\vect(Y).
\end{equation}
If $X$ is stable, then $\mathscr{A}(X)$ is invertible, by \eqref{equivalent_operator_matrix} the inverse operator $\mathcal{L}_X^{-1}(\cdot)$ exists on $\mathbb{R}^{n\times n}$.

In Procedure \ref{procedure_policy_itration}, suppose $K_1=\mathscr{K}(P_0)$,
where $P_0\in\mathbb{S}^{n}$ is chosen such that $K_1$ is stabilizing. Such a $P_0$ always exists. For example, since $K^*$ is stabilizing, one can choose $P_0$ close to $P^*$ by continuity. Then from \eqref{RPI_PE} and \eqref{RPI_PI}, the sequence $\{P_i\}_{i=0}^{\infty}$ generated by Procedure \ref{procedure_policy_itration} satisfies 
\begin{equation}\label{Kleinman_one_step}
    P_{i+1} = \mathcal{L}_{\mathcal{A}(\mathscr{K}(P_i))}^{-1}\left(- S - \mathscr{K}(P_i)^TR\mathscr{K}(P_i)\right).
\end{equation}
If $P_i$ is regarded as the state, and the iteration index $i$ is regarded as time, then \eqref{Kleinman_one_step} is a discrete-time dynamical system and $P^*$ is an equilibrium by Theorem \ref{theorem_standard_PI}. The next lemma shows that $P^*$ is actually a locally exponentially stable equilibrium, whose proof is given in Appendix B. %\ref{proof_locally_exp_stab}.
\begin{lemma}\label{lemma_local_exponentail_stable}
    For any $\sigma<1$, there exists a $\delta_0(\sigma)>0$, such that for any $P_i\in\mathcal{B}_{\delta_0}(P^*)$, $\mathscr{R}(P_i)$ is invertible, $\mathcal{A}(\mathscr{K}(P_i))$ is stable and
    $\Vert P_{i+1}-P^* \Vert_F \leq \sigma\Vert P_i-P^*\Vert_F$.
\end{lemma}
Lemma \ref{lemma_local_exponentail_stable} is inspired by \cite[Theorem 2]{1099755}, which states that Procedure \ref{procedure_policy_itration} has the rate of convergence
\begin{equation}\label{rate_of_convergence}
    \Vert P_{i+1}-P^* \Vert_F\leq c_0\Vert P_{i}-P^* \Vert_F^2.
\end{equation}
for any $P_i\geq P^*$, and some $c_0>0$. Notice that Lemma \ref{lemma_local_exponentail_stable} does not have the requirement $P_i\geq P^*$.

In Procedure \ref{procedure_robust_policy_iteration}, suppose $\hat{K}_1 = \mathscr{K}(\tilde{P}_0)$ and $\Delta G_0=0$, where $\tilde{P}_0\in\mathbb{S}^{n}$ is chosen such that $\hat{K}_1$ is stabilizing. If $\hat{K}_i$ is stabilizing and $[\hat{G}_{i}]_{uu}$ is invertible for all $i\in\mathbb{Z}_+$, $i>0$ (this is possible under certain conditions, see Appendix C, the sequence $\{\tilde{P}_i\}_{i=0}^{\infty}$ generated by Procedure \ref{procedure_robust_policy_iteration} satisfies
\begin{equation}\label{robust_Kleinman_one_step}
\begin{split}
    \tilde{P}_{i+1} &= \mathcal{L}_{\mathcal{A}(\mathscr{K}(\tilde{P}_i))}^{-1}\left(- S - \mathscr{K}(\tilde{P}_i)^TR\mathscr{K}(\tilde{P}_i)\right)\\
    &+  \mathcal{E}(\tilde{G}_i,\Delta G_i),    
\end{split}
\end{equation}
where
\begin{equation*}
    \begin{split}
        \mathcal{E}(\tilde{G}_i,\Delta G_i) &= \mathcal{L}_{\mathcal{A}(\hat{K}_{i+1})}^{-1}\left(-S-\hat{K}^T_{i+1}R\hat{K}_{i+1}\right) \\
        &- \mathcal{L}_{\mathcal{A}(\mathscr{K}(\tilde{P}_i))}^{-1}\left(- S - \mathscr{K}(\tilde{P}_i)^TR\mathscr{K}(\tilde{P}_i)\right).
    \end{split}
\end{equation*}
Here, the dependence of $\mathcal{E}$ on $\tilde{G}_i$ and $\Delta G_i$ comes from \eqref{eRPI_PI}. Regarding $\{\Delta G_i\}_{i=0}^{\infty}$ as the disturbance input, the next lemma shows that dynamical system \eqref{robust_Kleinman_one_step} is locally input-to-state stable, whose proof can be found in Appendix C.
\begin{lemma}\label{lemma_local_ISS}
For $\sigma$ and its associated $\delta_0$ in Lemma \ref{lemma_local_exponentail_stable}, there exists $\delta_1(\delta_0)>0$, such that if $\Vert\Delta G\Vert_\infty<\delta_1$, $\tilde{P}_0\in\mathcal{B}_{\delta_0}(P^*)$,
\begin{enumerate}[label=(\roman*),leftmargin=0.8cm]
    \item\label{lemma_local_ISS_item_well_defined} $[\hat{G}_{i}]_{uu}$ is invertible, $\hat{K}_i$ is stabilizing, $\forall i\in\mathbb{Z}_+$, $i>0$;
    \item\label{lemma_local_ISS_item_local_ISS} \eqref{robust_Kleinman_one_step} is locally input-to-state stable (see Definition \ref{Definition_local_ISS}):
    $$\Vert \tilde{P}_i - P^*\Vert_F\leq \beta(\Vert\tilde{P}_0-P^*\Vert_F,i) + \gamma(\Vert \Delta G\Vert_\infty),$$
    for all $i\in\mathbb{Z}_+$, where
    $\beta(y,i) = \sigma^iy$, $\gamma(y) = c_3y/(1-\sigma)$, $y\in\mathbb{R}$
    and $c_3(\delta_0)>0$.
    \item\label{lemma_local_ISS_item_K_bounded} $\Vert\hat{K}_i\Vert_F<\kappa_1$ for some $\kappa_1\in\mathbb{R}_+$, $\forall i\in\mathbb{Z}_+$, $i>0$;
    \item\label{lemma_local_ISS_item_converging_input_converging_state} $\lim_{i\rightarrow\infty} \Vert\Delta G_i\Vert_F = 0$ implies $\lim_{i\rightarrow\infty} \Vert \tilde{P}_i-P^* \Vert_F=0$.
\end{enumerate}
\end{lemma}
To prove Lemma \ref{lemma_local_ISS}, we firstly prove that with the given conditions, by continuity $[\hat{G}_{i}]_{uu}$ is invertible, $\hat{K}_i$ is stabilizing and
$\Vert\mathcal{E}(\tilde{G}_i,\Delta G_i)\Vert_F\leq c_3\Vert\Delta G_i\Vert_F$.
Then by Lemma \ref{lemma_local_exponentail_stable} and \eqref{robust_Kleinman_one_step}, if $\tilde{P}_i\in\mathcal{B}_{\delta_0}(P^*)$, $\delta_1$ can be chosen small enough so that
\begin{align}
    \Vert \tilde{P}_{i+1}-P^*\Vert_F&\leq \sigma\Vert \tilde{P}_{i}-P^*\Vert_F + c_3\Vert \Delta G\Vert_\infty  \label{proof_illustrate_exp_stab_robust}  \\
    &<\sigma\delta_0 + c_3\delta_1<\delta_0 \nonumber.
\end{align}
By mathematical induction, unrolling \eqref{proof_illustrate_exp_stab_robust} completes the proof. In the unrolling process, the coefficient $0<\sigma<1$ in the exponential stability of Lemma \ref{lemma_local_exponentail_stable} prevents the accumulated effects of disturbance $\Delta G_i$ from driving $\Vert \tilde{P}_{i+1}-P^*\Vert_F$ to the infinity.

Intuitively, Lemma \ref{lemma_local_ISS} implies that in Procedure \ref{procedure_robust_policy_iteration}, if $\tilde{P}_0$ is near $P^*$ (thus $\hat{K}_1$ is near $K^*$), and the disturbance input $\Delta G$ is bounded and not too large, then the cost of the generated control policy $\hat{K}_i$ is also bounded, and will ultimately be no larger than a constant proportional to the $l^\infty$-norm of the disturbance. The smaller the disturbance is, the better the ultimately generated policy is. In other words, the algorithm described in Procedure \ref{procedure_robust_policy_iteration} is not sensitive to small disturbances when the initial condition is in a neighbourhood of the optimal solution.

The requirement that the initial condition $\tilde{P}_0$ needs to be in a neighbourhood of $P^*$ in Lemma \ref{lemma_local_ISS} can be removed, as stated in the following theorem whose proof is given in the Appendix D. 
\begin{theorem}\label{theorem_global}
For any given stabilizing control gain $\hat{K}_1$ and any $\epsilon>0$, if $S>0$, there exist $\delta_2(\epsilon,\hat{K}_1)>0$, $\Pi(\delta_2)>0$, and $\kappa(\delta_2)>0$, such that for all $\Delta G$ satisfying $\Vert \Delta G \Vert_\infty < \delta_2$, $[\hat{G}_{i}]_{uu}$ is invertible, $\hat{K}_i$ is stabilizing, $\Vert \tilde{P}_i\Vert_F<\Pi$, $\Vert \hat{K}_i\Vert_F<\kappa$, $\forall i\in\mathbb{Z}_+$, $i>0$ and
$$\limsup_{i\rightarrow\infty} \Vert \tilde{P}_i-P^* \Vert_F<\epsilon.$$
If in addition $\lim_{i\rightarrow\infty} \Vert\Delta G_i\Vert_F = 0$, then $\lim_{i\rightarrow\infty} \Vert \tilde{P}_i-P^* \Vert_F=0$.
\end{theorem}
Here are the essential elements of the proof for Theorem \ref{theorem_global}: It is firstly proved that given any stabilizing control gain $\hat{K}_1$, there exist $\bar{i}\in\mathbb{Z}_+$, $\bar{i}<+\infty$, and $b_{\bar{i}}>0$, such that if $\Vert\Delta G_i\Vert_F< b_{\bar{i}}$ for $i = 1,2,\cdots,\bar{i}$, then (1) $[\hat{G}_{i}]_{uu}$ is invertible, $\hat{K}_i$ is stabilizing and bounded, $\tilde{P}_i$ is bounded, $i = 1,2,\cdots,\bar{i}$, (2) $\tilde{P}_{\bar{i}}$ enters the neighbourhood of $P^*$, i.e., $\mathcal{B}_{\delta_0}(P^*)$ defined in Lemma \ref{lemma_local_ISS}. Secondly, an application of Lemma \ref{lemma_local_ISS} completes the proof.

In Theorem \ref{theorem_global}, $\hat{K}_1$ can be any stabilizing control gain, which is different from that of Lemma \ref{lemma_local_ISS}. When there is no disturbance, Theorem \ref{theorem_global} implies the convergence result of Procedure \ref{procedure_policy_itration} in \cite[Theorem 1]{1099755} (i.e. Theorem \ref{theorem_standard_PI} in this paper).

\section{Optimistic Least-Squares Policy Iteration}\label{section_datadriven_PI}
For system \eqref{LTI_sys_stochastic}, due to the presence of stochastic noise $w_k$, the cost function \eqref{cost_deterinistic} will not be finite. Thus alternatively the objective is to find a control law in the form of $u = -Kx$ directly from the input/state data, minimizing the cost function
\begin{equation}\label{cost_average}
    J_{Avg}(u) = \lim_{N\rightarrow\infty}\frac{1}{N}\underset{\substack{w_k\\k=0,1,\cdots}}{\mathbb{E}}\left\{\sum_{k=0}^{N-1}c(x_k,u_k)\right\},
\end{equation}
where $S$ and $R$ in $c(x_k,u_k)$ are positive definite. It is well-known \cite[Section 4.4]{BertsekasOptimalControl2} that this problem shares the same optimal solutions with the standard LQR for system \eqref{LTI_sys} and cost function \eqref{cost_deterinistic}. Specifically, the optimal control gain is given by \eqref{optimal_gain}, and the optimal cost is $J^*_{Avg} = \trace(C^TP^*C)$, with $P^*$ the unique positive definite solution of \eqref{ARE}. For any stabilizing gain $K$, the cost it induces is $J_{Avg}(-Kx) = \trace(C^TP_KC)$, with $K$ and $P_K$ satisfying \eqref{algebraic_Lyapunov_equation} (or equivalently \eqref{RPI_PE}). Note that the assumption that $w_k\sim\mathcal{N}(0,I_q)$ in \eqref{LTI_sys_stochastic} is not a restriction, since any random variable $X_1\sim \mathcal{N}(0,\Sigma)$ with $\Sigma\in\mathbb{S}^q$ positive semidefinite, can be represented by $X_1=DX_2$, where $\Sigma=D^TD$, $D\in\mathbb{R}^{q\times q}$, and $X_2\sim \mathcal{N}(0,I_q)$. Then $D$ is absorbed into $C$ in \eqref{LTI_sys_stochastic}.

The optimistic least-squares policy iteration (O-LSPI) is based on the following observation: for a stabilizing gain $K$, its associated $P_K$ is the stable equilibrium of linear dynamical system
\begin{equation}\label{policy_evaluation_model_based}
    P_{K,j+1} = \mathcal{H}(Q(P_{K,j}),K),\quad P_{K,0}\in\mathbb{S}^{n},
\end{equation}
where
\begin{equation}\label{Definition_Q}
\begin{split}
    Q(P_{K,j}) &= \left[
                \begin{array}{c|c}
                [Q(P_{K,j})]_{xx} & [Q(P_{K,j})]^T_{ux} \\
                \hline
                [Q(P_{K,j})]_{ux} & [Q(P_{K,j})]_{uu}
        \end{array}
        \right] \\
        &\triangleq \left[\begin{array}{cc}
                S + A^TP_{K,j}A  & A^TP_{K,j}B \\
                B^TP_{K,j}A & R + B^TP_{K,j}B
        \end{array}\right].            
\end{split}
\end{equation}
This fact can be easily verified by rewriting and vectorizing \eqref{policy_evaluation_model_based} into its equivalent form
\begin{equation}\label{policy_evaluation_model_based_vector}
\begin{split}
    p_{K,j+1} &= \left((A-BK)^T\otimes(A-BK)^T\right)p_{K,j} \\
    &+ \vect(S+K^TRK),\quad p_{K,0}\in\mathbb{R}^{n^2},    
\end{split}
\end{equation}
where $p_{K,j} = \vect(P_{K,j})$. Since $(A-BK)$ is stable, $(A-BK)^T\otimes(A-BK)^T$ is also stable. Thus \eqref{policy_evaluation_model_based_vector} admits a unique stable equilibrium. So does \eqref{policy_evaluation_model_based} and the unique solution must be $P_K$ because
\begin{equation}\label{relationship_H_G}
\begin{split}
    Q(P_{K,j}) &= G(P_{K,j}) + \blkdiag(P_{K,j},0),\\
\mathcal{H}(G(P_{K,j}),K) &= \mathcal{H}(Q(P_{K,j}),K) - P_{K,j}.    
\end{split}
\end{equation}
This implies that instead of solving \eqref{RPI_PE}, we may utilize iteration \eqref{policy_evaluation_model_based} to achieve policy evaluation. It is not hard to recognize that \eqref{policy_evaluation_model_based} is actually the LQR version of the optimistic policy iteration in \cite{tsitsiklis2002convergence,bertsekas2011approximate} for problems with discrete state and action spaces (thus the name ``optimistic'' in O-LSPI). Suppose a behavior policy (the policy used to generate data is called the behavior policy, see \cite{sutton2018reinforcement}) $u_k = -\hat{K}_1x_k + v_k$ is applied to the system to collect data, where $\hat{K}_1$ is stabilizing and $v_k$ is drawn i.i.d. from Gaussian distribution $\mathcal{N}(0,\sigma^2_u I_m)$, $\sigma_u\in\mathbb{R}_+$. Then the state-control pair $[x^T,u^T]^T$ admits a unique invariant distribution $\pi$. We make the following assumption. 
\begin{assumption}\label{assumption_PE}
$\mathbb{E}_{\pi}\left[\tilde{z}\tilde{z}^T\right]$ is invertible, where $$z = [x^T,u^T,1]^T.$$
\end{assumption}
For any $P\in\mathbb{S}^{n}$, we have
\begin{equation*}
    \begin{split}
        % &\mathbb{E}\left[x^T_{k+1}Px_{k+1}\vert x_k,u_k\right] \\
        % &= \mathbb{E}\left[(Ax_k + Bu_k +Cw_k)^TP(Ax_k + Bu_k +Cw_k)\vert x_k,u_k\right]  \\
        % &=\mathbb{E}\left[(Ax_k + Bu_k)^TP(Ax_k + Bu_k) + \trace(C^TPC) \right.\\
        % &\left.+ x_k^TSx_k + u^T_kRu_k\vert x_k,u_k\right] - c(x_k,u_k) \\
        \mathbb{E}\left[z^T_k F(P) z_k - x^T_{k+1}Px_{k+1}\vert x_k,u_k\right] = c(x_k,u_k)
    \end{split}
\end{equation*}
where $F(P) = \blkdiag(Q(P),\trace(C^TPC))$.
Vectorizing and multiplying the above equation by $\tilde{z}_k$ yields
\begin{equation*}
    \begin{split}
        &\mathbb{E}\left[\tilde{z}_k\tilde{x}^T_{k+1}\vert x_k,u_k\right]\svec(P) \\
        &= \mathbb{E}\left[\tilde{z}_k\tilde{z}^T_k\vert x_k,u_k\right]\svec(F(P)) 
        - \tilde{z}_kc(x_k,u_k).
    \end{split}
\end{equation*}
Taking expectation with respect to the invariant distribution $\pi$, by Assumption \ref{assumption_PE} we obtain
\begin{equation}\label{exact_param_estimation}
    \svec(F(P)) = \varphi_1^{-1}\left(\varphi_2\svec(P) + \varphi_3\right),
\end{equation}
where $\varphi_1 = \mathbb{E}_{\pi}\left[\tilde{z}_k\tilde{z}^T_k\right]$, $\varphi_2 = \mathbb{E}_{\pi}\left[\tilde{z}_k\tilde{x}^T_{k+1}\right]$, and $\varphi_3 = \mathbb{E}_{\pi}\left[\tilde{z}_kc(x_k,u_k)\right]$.
For known $P$, $F(P)$ can be estimated using least squares from the collected data
$$\svec(\hat{F}(P)) = \Phi_M^\dagger\Psi_M\svec(P) + \Phi_M^\dagger\Xi_M,$$
where $M\in\mathbb{Z}_+$, $M>0$ and
\begin{equation*}
    \begin{split}
        \Phi_M &= \frac{1}{M}\sum_{k=0}^{M-1}\tilde{z}_k\tilde{z}^T_k,\ \Psi_M = \frac{1}{M}\sum_{k=0}^{M-1}\tilde{z}_k\tilde{x}^T_{k+1}, \\
        \Xi_M &= \frac{1}{M}\sum_{k=0}^{M-1}\tilde{z}_kc(x_k,u_k).
    \end{split}
\end{equation*}
In this way, \eqref{policy_evaluation_model_based} can be solved approximately and directly from the data by noticing that $Q(P)=\mathcal{H}(F(P),0)$. The O-LSPI is presented in Algorithm \ref{algorithm_off_policy_stochastic}. Note that the same data matrices $\Phi_M$, $\Psi_M$ and $\Xi_M$ are reused for all iterations, thus O-LSPI is off-policy. The convergence of O-LSPI is proved in the following theorem.
\begin{theorem}\label{theorem_stochastic_algorithm_convergence}
In Algorithm \ref{algorithm_off_policy_stochastic}, under Assumption \ref{assumption_PE}, for any initial stabilizing control gain $\hat{K}_1$ and any $\epsilon>0$, there exist $T_0\in\mathbb{Z}_+$ and $M_0\in\mathbb{Z}_+$, such that for any $T\geq T_0$ and $M\geq M_0$, almost surely,
$$\limsup_{N\rightarrow\infty}\Vert \tilde{P}_N-P^*\Vert_F<\epsilon$$
and $\hat{K}_i$ is stabilizing for all $i=1,\cdots,N$, where $\tilde{P}_N$ is the unique solution of \eqref{algebraic_Lyapunov_equation} for $\hat{K}_N$.
\end{theorem}
The proof of Theorem \ref{theorem_stochastic_algorithm_convergence} can be found in Appendix E. Let $J_{Avg}(-\hat{K}_ix) = \trace(C^T\tilde{P}_{i}C)$ denote the true cost induced by $\hat{K}_i$. By Theorem \ref{theorem_global}, the task is to prove that there exist $T_0\in\mathbb{Z}_+$ and $M_0\in\mathbb{Z}_+$, such that for any $T\geq T_0$ and $M\geq M_0$, almost surely, $\Vert\Delta G\Vert_\infty< \delta_2$. For Algorithm \ref{algorithm_off_policy_stochastic}, 
\begin{equation*}
    \hat{G}_i = \hat{Q}_{i,T} - \blkdiag(\hat{P}_{i,T},0).
\end{equation*}
Using \eqref{relationship_H_G}, we have
\begin{align}
    \Vert \Delta G_i\Vert_F%&=\Vert \hat{G}_i - \tilde{G}_i\Vert_F\nonumber\\
    % &\leq \Vert \hat{Q}_{i,T}-Q(\tilde{P}_i)\Vert_F + \Vert \hat{P}_{i,T}-\tilde{P}_i\Vert_F   \nonumber\\
    &\leq \Vert \hat{Q}_{i,T} - Q(\hat{P}_{i,T})\Vert_F + \Vert Q(\hat{P}_{i,T}) - Q(\tilde{P}_i)\Vert_F \nonumber\\
    &+ \Vert \hat{P}_{i,T}-\tilde{P}_i\Vert_F. \label{error_decomposition}
\end{align}
Since $\hat{K}_1$ is stabilizing, by the Birkhoff ergodic theorem \cite[Theorem 16.14]{korlov2007theory}, almost surely
\begin{equation}\label{ergodic_converge}
\begin{split}
    \lim_{M\rightarrow\infty} \Phi_M &= \varphi_1, \lim_{M\rightarrow\infty} \Psi_M = \varphi_2,  \\
    \lim_{M\rightarrow\infty} \Xi_M &= \varphi_3.
\end{split}
\end{equation}
Using \eqref{ergodic_converge}, Assumption \ref{assumption_PE}, \eqref{policy_evaluation_model_based} and \eqref{exact_param_estimation}, we are able to show that there exist $T_0$ and $M_0$, independent of iteration index $i$, such that for any $T\geq T_0$ and $M\geq M_0$, almost surely every term in \eqref{error_decomposition} is less than $\delta_2/3$. Then Theorem \ref{theorem_global} completes the proof.
\begin{algorithm}
\KwIn{Initial stabilizing control gain $\hat{K}_1$, Number of policy iterations $N$, Number of iterations for policy evaluation $T$, Number of rollout $M$, Exploration variance $\sigma^2_u$.}
Collect data with input $u_k=-\hat{K}_1x_k+v_k$, $v_k\sim \mathcal{N}(0,\sigma^2_u I_m)$, to construct $\Phi_M$, $\Psi_M$ and $\Xi_M$\;
\For{$i=1,\cdots,N-1$}{
$\hat{P}_{i,0}\leftarrow 0$\;
\For{$j=0,\cdots,T-1$}{
$\svec(\hat{F}_{i,j})\leftarrow \Phi_M^\dagger\Psi_M\svec(\hat{P}_{i,j}) + \Phi_M^\dagger\Xi_M$\;\label{algorithm_policy_evaluation_begin}
$\hat{Q}_{i,j}\leftarrow \mathcal{H}(\hat{F}_{i,j},0)$\;
$\hat{P}_{i,j+1} \leftarrow \mathcal{H}(\hat{Q}_{i,j},\hat{K}_i)$\;\label{algorithm_policy_teration_end}
}
$\svec(\hat{F}_{i,T})\leftarrow \Phi_M^\dagger\Psi_M\svec(\hat{P}_{i,T}) + \Phi_M^\dagger\Xi_M$\;\label{algorithm_estimate_final_Q_1}
$\hat{Q}_{i,T}\leftarrow \mathcal{H}(\hat{F}_{i,T},0)$\;\label{algorithm_estimate_final_Q_2}
$\hat{K}_{i+1}\leftarrow [\hat{Q}_{i,T}]_{uu}^{-1}[\hat{Q}_{i,T}]_{ux}$\;
}
\KwRet{$\hat{K}_N$.}
\caption{O-LSPI\label{algorithm_off_policy_stochastic}}
\end{algorithm}

\section{Experiments}\label{section_simulation}
We apply O-LSPI to the LQR problem studied in \cite{krauth2019finite} with
\begin{equation*}
\begin{split}
A &= \left[\begin{array}{ccc}
    0.95 & 0.01 & 0 \\
    0.01 & 0.95 & 0.01 \\
    0 & 0.01 & 0.95
\end{array}\right],\quad B=\left[\begin{array}{cc}
    1 & 0.1 \\
    0 & 0.1 \\
    0 & 0.1
\end{array}\right],  \\
C &= S =I_3,\quad R = I_2.
\end{split}
\end{equation*}
Here $A$ is stable, so we just choose the initial stabilizing control gain to be $\hat{K}_1 = 0_{2\times 3}$. The exploration variance is set to $\sigma_u^2=1$. All the experiments are conducted using MATLAB\footnote{https://www.mathworks.com/} 2017b, on the New York University High Performance Computing Cluster \textit{Prince} with $4$ CPUs and 16GB Memory. Algorithm \ref{algorithm_off_policy_stochastic} is implemented with increasing values of parameters $N$, $T$ and $M$, until the performance of the resulting control gain (almost) does not improve. This yields $N=5$, $T=45$ and $M=10^6$. To investigate the performance of the algorithm with different values of $M$ and $T$, we conducted two sets of experiments: (a) Fix $N=5$ and $T=45$, and implement Algorithm \ref{algorithm_off_policy_stochastic} with increasing values of $M$ from $200$ to $10^6$; (b) Fix $N=5$ and $M=10^6$, and implement Algorithm \ref{algorithm_off_policy_stochastic} with increasing values of $T$ from $2$ to $45$. To evaluate the stability, we run Algorithm \ref{algorithm_off_policy_stochastic} for 100 times per set of parameters, and compute the fraction of times it produces stable policies in all phases (left column in Figure \ref{exp_results}). To evaluate the optimality, the relative error of the cost function $\trace(C^T(\tilde{P}_N-P^*)C)/\trace(C^TP^*C)$ is calculated. The relative errors of 100 stable implementations of Algorithm 1 are collected (i.e., implementation that yields stabilizing control gains in all phases), based on which the sample average (middle column in Figure \ref{exp_results}) and sample variance (right column in Figure \ref{exp_results}) of the relative error are plotted.

In Figure \ref{exp_results}, as the number of rollout $M$ increases, the fraction of stability becomes one, and both the sample average and sample variance of relative error converge to zero. The fraction of stability is not sensitive to the number of iteration for policy evaluation $T$. But as $T$ increases, the sample average and sample variance of relative error improve and converge to zeros. These observations are consistent with our Theorem \ref{theorem_stochastic_algorithm_convergence}, thus are also consistent with our robustness analysis for policy iteration, since Theorem \ref{theorem_stochastic_algorithm_convergence} is based on Theorem \ref{theorem_global}.

For comparison, the off-policy least-squares policy iteration algorithm LSPIv1 in \cite{krauth2019finite} is also implemented, using the same setting with the first set of experiments of various $M$ (upper row in Figure \ref{exp_results}). The O-LSPI and LSPIv1 have similar performance for $M\geq 10^4$, while the performance of LSPIv1 is slightly better than that of O-LSPI for $M<10^4$. This may be explained by the fact that the LSPIv1 in \cite{krauth2019finite} assume knowledge of the matrix $C$ in \eqref{LTI_sys_stochastic}, which is not required in O-LSPI.

\begin{figure*}[!htb]
    \centering
    \includegraphics[width=0.7\linewidth]{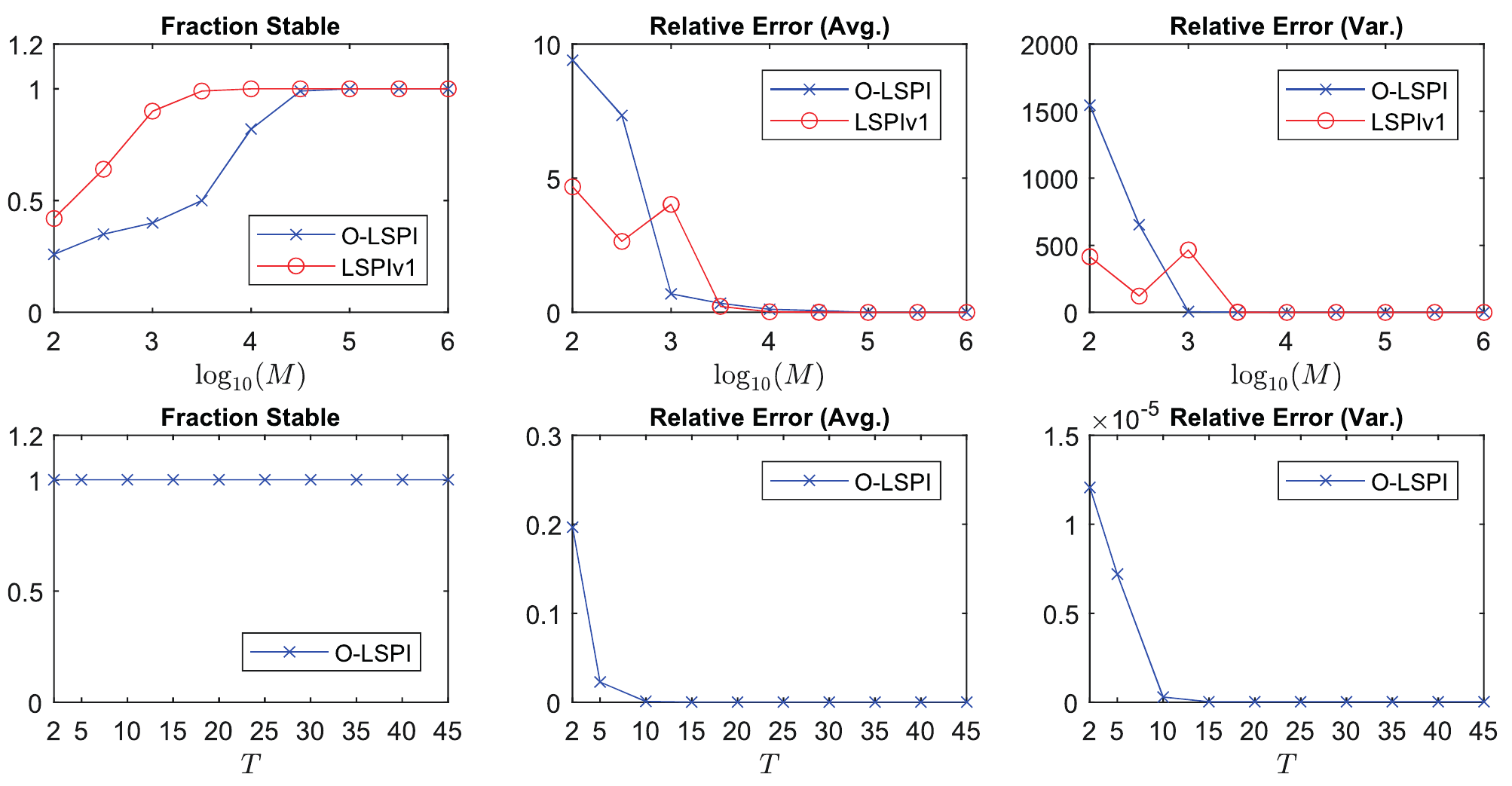}
    \caption{Experimental evaluation on the dynamics of \cite{krauth2019finite}.}\label{exp_results}
\end{figure*}

Finally, the performance of O-LSPI and LSPIv1 with various choices of exploration variance $\sigma^2_u$ has been investigated on the same example (see Appendix F). The performance of both O-LSPI and LSPIv1 is best when the exploration variances are large ($\sigma^2_u\geq 10$). The performance of both O-LSPI and LSPIv1 deteriorates when the exploration noise variances are medium and small ($\sigma^2_u< 10$). And O-LSPI performs better than LSPIv1 when the exploration noise variances are small ($\sigma^2_u \leq 10^{-5}$).

\section{Related Work}
The investigation of the robustness of policy iteration for problems with continuous state/control spaces is available in previous literature. In \cite[Proposition 3.6]{BertsekasOptimalControl2}, for discounted optimal control problems of discrete-time systems, it is reported that
\begin{equation}\label{Bound_Bertsekas}
    \limsup_{i\rightarrow\infty}\max_{x\in \mathbb{X}\subset\mathbb{R}^n}\left(J_{\mu^i}(x)-J^*(x)\right)\leq \frac{\epsilon +2\alpha\delta}{(1-\alpha)^2},
\end{equation}
where $\mu^i$ is the policy generated in $i$th iteration, $\delta$ and $\epsilon$ are the upper bounds of the errors in policy evaluation and policy improvement respectively, $0<\alpha<1$ is the discount factor. Bound \eqref{Bound_Bertsekas} and our bound in Theorem \ref{theorem_global} have the similar styles. However, in our setting the discount factor is $\alpha=1$ so our bound cannot be implied by \eqref{Bound_Bertsekas}. Utilizing the fact that Riccati operator is contractive in Thompson's part metric \cite{thompson1963certain}, it is shown in \cite[Appendix B]{krauth2019finite} that the convergence to the optimal solutions is still achieved in Thompson's part metric, if the errors converge to zero. But it is unclear if this result could imply Theorem \ref{theorem_global} in this paper. Sufficient conditions on the errors are given in \cite[Chapter 2]{hylla2011extension} and \cite[Chapter 2]{boussios1998approach} for continuous-time linear and nonlinear system dynamics respectively, to guarantee that the newly generated control policy is stabilizing and improved. The robustness analysis in this paper is parallel to that in \cite{pang2020robust}. However, since we are dealing with discrete-time systems here, the derivations and proofs are inevitably distinct.

In recent years, there have been resurgent research interests in LQR problems, about learning the optimal solutions from the input/state/output data. The model-based certainty equivalence methods explicitly estimate the values of $A$, $B$ and $C$ in \eqref{LTI_sys_stochastic} from data, and obtain near-optimal solutions based on the estimations, see \cite{pmlr-v19-abbasi-yadkori11a,ouyang2017learning,pmlr-v80-abeille18a,NIPS2018_7673,SHIRANIFARADONBEH2020108982,umenberger2020optimistic,cassel2020logarithmic,basei2020linear}, to name a few. The model-free methods aim at finding the near-optimal solutions directly from the data, without the estimations of system dynamics. Action-value model-free methods learn the value functions of policies, and then generate new (improved) policies based on the estimated value functions, see \cite{735224,pmlr-v80-tu18a,krauth2019finite,pmlr-v89-abbasi-yadkori19a,pmlr-v99-tu19a,doi:10.1137/18M1214147}. Policy-gradient model-free methods directly learn the policies based on the gradient of some scalar performance measure with respect to the policy parameter, see \cite{pmlr-v80-fazel18a,bu2019lqr,preiss2019analyzing,mohammadi2019convergence,yang2019global,qu2020combining,jansch2020convergence,jansch2020policy,pmlr-v120-furieri20a}. Derivative-free model-free methods randomly search in the parameter space of policies for the near-optimal solutions, without explicitly estimate the gradient, see \cite{NIPS2018_7451,pmlr-v89-malik19a,li2019distributed}. 

Most of the model-free methods for LQR mentioned above are on-policy, fewer theoretical results exist for off-policy methods. Among the off-policy action-value model-free methods for LQR, the most related to our proposed O-LSPI are the LSPIv1 in \cite{krauth2019finite} and the MFLQv1 in \cite{pmlr-v89-abbasi-yadkori19a}. However, (a) no convergence result is reported for LSPIv1 in \cite{krauth2019finite}, and (b) MFLQv1 in \cite{pmlr-v89-abbasi-yadkori19a} needs to learn the $P_K$ in \eqref{algebraic_Lyapunov_equation} first in on-policy fashion, before it can learn the $Q(P_K)$ in \eqref{Definition_Q} in off-policy fashion in each iteration, and (c) both the LSPIv1 and MFLQv1 need the knowledge of matrix $C$ in \eqref{LTI_sys_stochastic}, and (d) both the LSPIv1 and MFLQv1 need to solve a pseudo-inverse problem in each iteration. In contrast, in our O-LSPI, (a) a convergence result is given (Theorem \ref{theorem_stochastic_algorithm_convergence}), and (b) both the $P_K$ and $Q(P_K)$ are learned in off-policy fashion (Lines \ref{algorithm_policy_evaluation_begin} to \ref{algorithm_policy_teration_end} in Algorithm \ref{algorithm_off_policy_stochastic}), and (c) no knowledge of $C$ is required, and (d) the pseudo-inverse problem only needs to be solved once.

Finally, it is worth mentioning again that the robustness of RL to errors in the learning processes is analyzed in this paper. This is different from the the robustness of the controllers learned by RL to disturbances in the system dynamics, studied in \cite{9254115,zhang2020on,zhang2020policy,9197000,jiang2017robust}. We also notice that using control theory to study RL algorithms, as we did in this paper, has become popular recently. By regarding the learning processes as dynamical systems, abundant results and techniques in control theory can be applied to obtain better understanding of RL algorithms, see \cite{pmlr-v99-srikant19a,NEURIPS2019_e354fd90,NEURIPS2019_e2c4a40d,lee2019unified,doi:10.1137/18M1214147}.

\section{Concluding Remarks}\label{section_conclusion}
This paper analyzes the robustness of policy iteration for discrete-time LQR. It is proved that starting from any stabilizing initial policy, the solutions generated by policy iteration with errors are bounded and ultimately enter and stay in a neighbourhood of the optimal solution, as long as the errors are small and bounded. This result in the spirit of small-disturbance input-to-state stability is employed to prove the convergence of the optimistic least-squares policy iteration (O-LSPI), a novel off-policy model-free RL algorithm for discrete-time LQR with additive stochastic noises in the dynamics. The theoretical results are verified by the experiments on a numerical example.

                                  % the text height of the last page by a suitable amount.

%%%%%%%%%%%%%%%%%%%%%%%%%%%%%%%%%%%%%%%%%%%%%%%%%%%%%%%%%%%%%%%%%%%%%%%%%%%%%%%%

%%%%%%%%%%%%%%%%%%%%%%%%%%%%%%%%%%%%%%%%%%%%%%%%%%%%%%%%%%%%%%%%%%%%%%%%%%%%%%%%

%%%%%%%%%%%%%%%%%%%%%%%%%%%%%%%%%%%%%%%%%%%%%%%%%%%%%%%%%%%%%%%%%%%%%%%%%%%%%%%%

\section*{Acknowledgments}
This work has been supported in part by the U.S. National Science Foundation under Grants ECCS-1501044 and EPCN-1903781. Bo Pang thanks Dr. Stephen Tu for sharing the code of the least-squares policy iteration algorithms in \cite{krauth2019finite}.

%%%%%%%%%%%%%%%%%%%%%%%%%%%%%%%%%%%%%%%%%%%%%%%%%%%%%%%%%%%%%%%%%%%%%%%%%%%%%%%%
\bibliography{RPIRef}
\onecolumn
\appendix
\setcounter{section}{0}

\section{Appendix A: Useful auxiliary results}
Some useful properties of $\mathcal{L}_X(\cdot)$ are provided below.
\begin{lemma}\label{Lyapunov_properties}
Suppose $X\in\mathbb{R}^{n\times n}$ is stable.
\begin{enumerate}[label=(\roman*)]
    \item\label{Lyapunov_solution} For each $Z\in\mathbb{R}^{n\times n}$, if $Y\in\mathbb{S}^n$ is the solution of equation $\mathcal{L}_X(Y)=-Z$, then
    $$Y=\sum_{k=0}^{\infty} (X^T)^kZX^k.$$
    \item\label{Lyapunov_monotone} $\mathcal{L}_{X}(Y_1)\leq \mathcal{L}_{X}(Y_2) \implies Y_1\geq Y_2$, where $Y_1,Y_2\in\mathbb{S}^n$.
\end{enumerate}
\end{lemma}
\begin{proof}
\ref{Lyapunov_solution} can be found in \cite{1099755}. For \ref{Lyapunov_monotone}, let $x$ be the solution of system $x_{k+1}=Xx_k$, $x(0)=x_0$. Then $x_k = X^kx_0$ and by \ref{Lyapunov_solution}
\begin{equation*}
    x_0^TY_1x_0=\sum_{k=0}^{\infty} x_k^T(-\mathcal{L}_{X}(Y_1))x_k\geq\sum_{k=0}^{\infty} x_k^T(-\mathcal{L}_{X}(Y_2))x_k=x_0^TY_2x_0.
\end{equation*}
$x_0$ is arbitrary, thus $Y_1\geq Y_2$.
\end{proof}
When Definition \ref{definition_exp_stab_nonlinear} is truncated to the linear dynamical system \eqref{LTI_sys}, since $x^*=0$, we have the following definition.
\begin{definition}\label{definition_exp_stab_linear}
System $x_{k+1}=Zx_k$ is exponentially stable if for some $a>0$ and $0<b<1$,
$$\Vert Z^k\Vert_2\leq ab^k.$$
\end{definition}
Definition \ref{definition_exp_stab_linear} is equivalent to the $(\tau,\rho)$-stability defined in \cite[Definition 1]{krauth2019finite}, and is closely related to the strong stability defined in \cite[Definition 5]{cassel2020logarithmic}.
% \begin{definition}[{\cite[Definition 5.4.1.]{agarwal2000difference}}]\label{definition_exp_stab}
% Square matrix $Z$ is exponentially stable, if
% $$\Vert Z^k\Vert_2\leq \alpha \exp(-\beta k),\quad \forall k\in\mathbb{Z}_+$$
% for some $\alpha>0$ and $\beta>0$.
% \end{definition}
\begin{lemma}[{\cite[Remark 5.5.3.]{agarwal2000difference}}]\label{lemma_LTI_exp_asym_stab}
System $x_{k+1}=Zx_k$ is stable if and only if it is exponentially stable.
\end{lemma}
Generally in Definition \ref{definition_exp_stab_linear}, $a$ and $b$ depend on the choice of matrix $Z$. The next lemma shows that a set of (exponentially) stable systems can share the same constants $a$ and $b$.
\begin{lemma}\label{uniform_exp_stab_over_compact_set}
Let $\mathcal{Z}$ be a compact set of stable matrices, then there exist an $a_0>0$ and a $0<b_0<1$, such that
$$\Vert Z^k\Vert_2\leq a_0 b_0^k,\quad \forall k\in\mathbb{Z}_+$$
for any $Z\in\mathcal{Z}$.
\end{lemma}
\begin{proof}
For each $Z\in\mathcal{Z}$, by \cite[Lemma B.1]{pmlr-v99-tu19a} there exist $r>0$, $a>0$ and $0<b<1$, such that
$$\Vert Z'^k\Vert_2\leq a b^k,\quad \forall k\in\mathbb{Z}_+$$
for all $\Vert Z'-Z\Vert_2<r$. Then the compactness of $\mathcal{Z}$ completes the proof.
\end{proof}
The following lemma provides the relationship between operations $\vect(\cdot)$ and $\svec(\cdot)$.
\begin{lemma}[{\cite[Page 57]{Magnus2007matrix}}]\label{lemma_relationship_vec_svec}
For $X\in\mathbb{S}^n$, there exists a unique matrix $D_n\in\mathbb{R}^{n^2\times \frac{1}{2}n(n+1)}$ with full column rank, such that
$$\vect(X) = D_n\svec(X),\quad \svec(X) = D_n^\dagger \vect(X).$$
$D_n$ is called the duplication matrix.
\end{lemma}
For $X\in\mathbb{R}^{n\times n}$, $Y\in\mathbb{R}^{n\times m}$, $X+\Delta X\in\mathbb{R}^{n\times n}$, $Y + \Delta Y\in\mathbb{R}^{n\times m}$, supposing $X$ and $X + \Delta X$ are invertible, the following inequality is repeatedly used
\begin{equation}\label{linear_perturbation}
\begin{split}
    &\Vert X^{-1}Y - (X+\Delta X)^{-1}(Y+\Delta Y)\Vert_F = \left\Vert X^{-1}Y - X^{-1}(Y+\Delta Y)\right. \\
    &\left.+ X^{-1}(Y+\Delta Y) - (X+\Delta X)^{-1}(Y+\Delta Y)\right\Vert_F\\
    &= \Vert -X^{-1}\Delta Y + X^{-1}\Delta X(X + \Delta X)^{-1}(Y+\Delta Y)\Vert_F\\
    &\leq \Vert X^{-1}\Vert_F\Vert \Delta Y\Vert_F + \Vert X^{-1}\Vert_F\Vert(X + \Delta X)^{-1}\Vert_F\Vert(Y+\Delta Y)\Vert_F\Vert\Delta X\Vert_F
\end{split}
\end{equation}

\section{Appendix B: Proof of Lemma \ref{lemma_local_exponentail_stable}}\label{proof_locally_exp_stab}
Since $\mathcal{A}(\mathscr{K}(P^*))$ is stable, by continuity there always exists a $\bar{\delta}_0>0$, such that $\mathscr{R}(P_i)$ is invertible, $\mathcal{A}(\mathscr{K}(P_i))$ is stable for all $P_i\in\mathcal{\bar{B}}_{\bar{\delta}_0}(P^*)$. Suppose $P_i\in\mathcal{\bar{B}}_{\bar{\delta}_0}(P^*)$. Subtracting
$$K_{i+1}^TB^TP^*A+A^TP^*BK_{i+1}-K_{i+1}^T\mathscr{R}(P^*)K_{i+1}$$
from both sides of the ARE \eqref{ARE} for $P=P^*$ and noticing that $K_{i+1} = \mathscr{K}(P_i)$, we have
\begin{equation}\label{ARE_reformulation}
\begin{split}
    &\mathcal{L}_{\mathcal{A}(\mathscr{K}(P_i))}(P^*) = - S - \mathscr{K}^T(P_i)R\mathscr{K}(P_i) + \\
    &(\mathscr{K}(P_i)-\mathscr{K}(P^*))^T\mathscr{R}(P^*)(\mathscr{K}(P_i)-\mathscr{K}(P^*)).
\end{split}
\end{equation}
Subtracting \eqref{ARE_reformulation} from \eqref{Kleinman_one_step}, we have
\begin{equation*}
    \begin{split}
        P_{i+1}-P^* = \mathcal{L}^{-1}_{\mathcal{A}(\mathscr{K}(P_i))}\left(((\mathscr{K}(P_i)-\mathscr{K}(P^*))^T\mathscr{R}(P^*)(\mathscr{K}(P_i)-\mathscr{K}(P^*))\right).
    \end{split}
\end{equation*}
Taking norm on both sides of above equation, \eqref{equivalent_operator_matrix} yields
\begin{equation*}
    \begin{split}
        \Vert P_{i+1}-P^* \Vert_F \leq \Vert \mathscr{A}(\mathcal{A}(\mathscr{K}(P_i)))^{-1}\Vert_2 \Vert \mathscr{R}(P^*)\Vert_F\Vert \mathscr{K}(P_i)-\mathscr{K}(P^*)\Vert_F^2.
    \end{split}
\end{equation*}
Since $\mathscr{K}(\cdot)$ is locally Lipschitz continuous at $P^*$, by continuity of matrix norm and matrix inverse, there exists a $c_1>0$, such that
$$\Vert P_{i+1}-P^* \Vert_F \leq c_1\Vert P_i - P^*\Vert_F^2,\quad \forall P_i\in\mathcal{\bar{B}}_{\bar{\delta}_0}(P^*).$$
So for any $0<\sigma<1$, there exists a $\bar{\delta}_0\geq\delta_0>0$ with $c_1\delta_0\leq \sigma$. This completes the proof.

\section{Appendix C: Proof of Lemma \ref{lemma_local_ISS}}\label{proof_local_ISS}
Before proving Lemma \ref{lemma_local_ISS}, we firstly prove some auxiliary lemmas. The Procedure \ref{procedure_robust_policy_iteration} will exhibit a singularity, if $[\hat{G}_{i}]_{uu}$ in \eqref{eRPI_PI} is singular, or the cost \eqref{cost_deterinistic} of $\hat{K}_{i+1}$ is infinity. The following lemma shows that if $\Delta G_i$ is small, no singularity will occur. Let $\bar{\delta}_0$ be the one defined in the proof of Lemma \ref{lemma_local_exponentail_stable}, then $\delta_0\leq\bar{\delta}_0$.
\begin{lemma}\label{lemma_stability_ISS}
There exists a $d(\delta_0)>0$ such that for all $\tilde{P}_i\in\mathcal{B}_{\delta_0}(P^*)$, $\hat{K}_{i+1}$ is stabilizing and $[\hat{G}_{i}]_{uu}$ is invertible, if $\Vert\Delta G_i\Vert_F\leq d$.
\end{lemma}
\begin{proof}
Since $\mathcal{\bar{B}}_{\bar{\delta}_0}(P^*)$ is compact and $\mathcal{A}(\mathscr{K}(\cdot))$ is a continuous function, set $$\mathcal{S} = \{\mathcal{A}(\mathscr{K}(\tilde{P}_i))\vert \tilde{P}_i\in\mathcal{\bar{B}}_{\bar{\delta}_0}(P^*)\}$$ is also compact. By continuity, for each $X\in\mathcal{S}$, there exists a $r(X)>0$ such that any $Y\in\mathcal{B}_{r(X)}(X)$ is stable. The compactness of $\mathcal{S}$ implies the existence of a $\underline{r}>0$, such that each $Y\in\mathcal{B}_{\underline{r}}(X)$ is stable for all $X\in\mathcal{S}$.
Similarly, there exists $d_1>0$ such that $[\hat{G}_{i}]_{uu}$ is invertible for all $\tilde{P}_i\in\mathcal{\bar{B}}_{\bar{\delta}_0}(P^*)$, if $\Vert\Delta G_i\Vert_F\leq d_1$. Note that in policy improvement step of Procedure \ref{procedure_policy_itration} (the policy update step in Procedure \ref{procedure_robust_policy_iteration}), the improved policy $\tilde{K}_{i+1}=[\tilde{G}_{i}]_{uu}^{-1}[\tilde{G}_{i}]_{ux}$ (the updated policy $\hat{K}_{i+1}$) is continuous function of $\tilde{G}_i$ ($\hat{G}_i$), and there exists a $0<d_2\leq d_1$, such that $\mathcal{A}(\hat{K}_{i+1})\in\mathcal{B}_{\underline{r}}(\mathcal{A}(\mathscr{K}(\tilde{P}_i)))$ for all $\tilde{P}_i\in\mathcal{\bar{B}}_{\bar{\delta}_0}(P^*)$, if $\Vert\Delta G_i\Vert_F\leq d_2$. Thus $\hat{K}_{i+1}$ is stabilizing. Setting $d=d_2$ completes the proof.
\end{proof}
By Lemma \ref{lemma_stability_ISS}, if $\Vert\Delta G_i\Vert_\infty\leq d$, then the sequence $\{\tilde{P}_i\}_{i=0}^{\infty}$ satisfies \eqref{robust_Kleinman_one_step}. For simplicity, we denote $\mathcal{E}(\tilde{G}_i,\Delta G_i)$ in \eqref{robust_Kleinman_one_step} by $\mathcal{E}_i$. The following lemma gives an upper bound on $\Vert \mathcal{E}_i\Vert_F$ in terms of $\Vert\Delta G_i\Vert_F$.
\begin{lemma}\label{lemma_total_error_bounded_ISS}
For any $\tilde{P}_i\in\mathcal{B}_{\delta_0}(P^*)$ and any $c_2>0$, there exists a $0<\delta_1^1(\delta_0,c_2)\leq d$, independent of $\tilde{P}_i$, where $d$ is defined in Lemma \ref{lemma_stability_ISS}, such that
$$\Vert \mathcal{E}_i\Vert_F\leq c_3\Vert \Delta G_i\Vert_F<c_2,$$ 
if $\Vert \Delta G_i\Vert_F<\delta_1^1$, where $c_3(\delta_0)>0$.
\end{lemma}
\begin{proof}
For any $\tilde{P}_i\in\bar{\mathcal{B}}_{\delta_0}(P^*)$, $\Vert \Delta G_i\Vert_F\leq d$, we have from \eqref{linear_perturbation}
\begin{align}
    \Vert \mathscr{K}(\tilde{P}_i) - \hat{K}_{i+1}\Vert_F
    %&\leq \Vert [\tilde{G}_{i}]^{-1}_{uu}[\tilde{G}_{i}]_{ux} - [\hat{G}_{i}]^{-1}_{uu}[\hat{G}_{i}]_{ux}\Vert_F  \nonumber\\
    %&\leq \Vert [\tilde{G}_{i}]^{-1}_{uu}[\tilde{G}_{i}]_{ux} - [\hat{G}_{i}]^{-1}_{uu}[\tilde{G}_{i}]_{ux} + [\hat{G}_{i}]^{-1}_{uu}[\tilde{G}_{i}]_{ux} - [\hat{G}_{i}]^{-1}_{uu}[\hat{G}_{i}]_{ux}\Vert_F \nonumber \\
    %&\leq \Vert [\hat{G}_{i}]^{-1}_{uu}\Vert_F\Vert\Delta G_i\Vert_F + \Vert [\tilde{G}_{i}]^{-1}_{uu} - [\hat{G}_{i}]^{-1}_{uu}\Vert_F\Vert[\tilde{G}_{i}]_{ux}\Vert_F \nonumber \\
    &\leq \Vert [\tilde{G}_{i}]^{-1}_{uu}\Vert_F\Vert\Delta G_i\Vert_F + \Vert[\hat{G}_{i}]^{-1}_{uu}\Vert_F \Vert[\tilde{G}_{i}]^{-1}_{uu}\Vert_F\Vert[\hat{G}_{i}]_{ux}\Vert_F\Vert \Delta G_i\Vert_F \nonumber \\
    &\leq c_4(\delta_0,d)\Vert \Delta G_i\Vert_F, \label{K_close}
\end{align}
where the last inequality comes from the continuity of matrix inverse and the extremum value theorem. Define
\begin{equation*}
\begin{split}
   \check{P}_{i} = \mathcal{L}_{\mathcal{A}(\hat{K}_{i+1})}^{-1}\left(-S-\hat{K}^T_{i+1}R\hat{K}_{i+1}\right),\quad \mathring{P}_{i} = \mathcal{L}_{\mathcal{A}(\mathscr{K}(\tilde{P}_i))}^{-1}\left(- S - \mathscr{K}(\tilde{P}_i)^TR\mathscr{K}(\tilde{P}_i)\right).    
\end{split}
\end{equation*}
Then by \eqref{equivalent_operator_matrix} and \eqref{robust_Kleinman_one_step},
\begin{equation*}
    \begin{split}
        \Vert\mathcal{E}_i\Vert_F &= \Vert \vect(\check{P}_{i}- \mathring{P}_{i})\Vert_2 \\
        \vect(\check{P}_{i}) &= \mathscr{A}^{-1}\left(\mathcal{A}(\hat{K}_{i+1})\right)\vect\left(-S-\hat{K}^T_{i+1}R\hat{K}_{i+1}\right) \\
        \vect(\mathring{P}_{i}) &= \mathscr{A}^{-1}\left(\mathcal{A}(\mathscr{K}(\tilde{P}_i))\right)\vect\left(- S - \mathscr{K}(\tilde{P}_i)^TR\mathscr{K}(\tilde{P}_i)\right).
    \end{split}
\end{equation*}
Define
\begin{equation*}
    \begin{split}
        \Delta \mathscr{A}_i &= \mathscr{A}\left(\mathcal{A}(\mathscr{K}(\tilde{P}_i))\right) - \mathscr{A}\left(\mathcal{A}(\hat{K}_{i+1})\right),\quad
        \Delta b_i = \vect\left(\mathscr{K}(\tilde{P}_i)^TR\mathscr{K}(\tilde{P}_i) - \hat{K}_{i+1}^TR\hat{K}_{i+1}\right).
    \end{split}
\end{equation*}
Using \eqref{K_close}, it is easy to check that $\Vert \Delta \mathscr{A}_i\Vert_F\leq c_5\Vert\Delta G_i\Vert_F$, $\Vert \Delta b_i\Vert_2\leq c_6\Vert\Delta G_i\Vert_F$, for some $c_5(\delta_0,d)>0$, $c_6(\delta_0,d)>0$. Then by \eqref{linear_perturbation}
\begin{equation*}
    \begin{split}
        \Vert\mathcal{E}_i \Vert_F&\leq \left\Vert \mathscr{A}^{-1}\left(\mathcal{A}(\hat{K}_{i+1})\right)\right\Vert_F \left(c_6 + c_5\left\Vert\mathscr{A}^{-1}\left(\mathcal{A}(\mathscr{K}(\tilde{P}_i))\right)\right\Vert_F\right.\\
        &\left.\times\left\Vert S + \mathscr{K}(\tilde{P}_i)^TR\mathscr{K}(\tilde{P}_i)\right\Vert_F\right)\Vert \Delta G_i\Vert_F \\
        &\leq c_3(\delta_0)\Vert \Delta G_i\Vert_F
    \end{split}
\end{equation*}
where the last inequality comes from the continuity of matrix inverse and Lemma \ref{lemma_stability_ISS}. Choosing $0<\delta^1_1\leq d$ such that $c_3\delta^1_1<c_2$ completes the proof.

%By continuity and Lemma \ref{lemma_local_exponentail_stable}, there exists $d\geq\delta_2>0$, such that for all $\tilde{P}_i\in\bar{\mathcal{B}}_{\delta_0}(P^*)$, $\Vert \Delta G_i\Vert_F<\delta_2$,
%$$\left\Vert \mathscr{A}^\dagger\left(\mathcal{A}(\mathscr{K}(\tilde{P}_i))\right)\right\Vert_2<\alpha_1,\quad \alpha_1\Vert\Delta \mathscr{A}_i\Vert_2<\alpha_2<1,\quad \Vert\mathring{P}_i\Vert_2<\alpha_3, \quad c_3\delta_2<c_2,$$
%where $\alpha_j$, $j=1,2,3$ are positive constants, $c_3 = (1-\alpha_2)^{-1}\alpha_1\left(\alpha_3c_5 + c_6\right)$.
%By applying standard linear system perturbation analysis \cite[Theorem 1.4.6.]{doi:10.1137/1.9781611971484}, we have
% \begin{equation*}
%     \begin{split}
%         \Vert \mathcal{E}_i\Vert_F&\leq \left(1-\left\Vert \mathscr{A}^\dagger\left(\mathcal{A}(\mathscr{K}(\tilde{P}_i))\right)\right\Vert_2\Vert\Delta \mathscr{A}_i\Vert_2\right)^{-1}\left\Vert \mathscr{A}^\dagger\left(\mathcal{A}(\mathscr{K}(\tilde{P}_i))\right)\right\Vert_2 \\
%         &\times\left(\Vert \mathring{P}_i\Vert_2\Vert\Delta \mathscr{A}_i\Vert_2 + \Vert\Delta b_i\Vert_2\right) \\
%         &\leq (1-\alpha_2)^{-1}\alpha_1\left(\alpha_3c_5 + c_6\right)\Vert\Delta G_i\Vert_F = c_3 \Vert\Delta G_i\Vert_F<c_2,
%     \end{split}
% \end{equation*}
\end{proof}
%Since the set of all $P\in\mathbb{R}^{n\times n}$ such that $\mathcal{A}(P)$ is Hurwitz is open, and $\mathcal{\bar{B}}_{\bar{\delta}_0}(P^*)$ is compact, there exists a $\hat{\delta}_0>\bar{\delta_0}$, such that $P$ is Hurwitz for all $P\in\mathcal{\bar{B}}_{\hat{\delta}_0}(P^*)$. Now we derive an upper bound of $\mathcal{E}(\tilde{P}_i,\Delta G_i)$ with respect to $\Delta G_i$. 
Now we are ready to prove Lemma \ref{lemma_local_ISS}.
\begin{proof}[Proof of Lemma \ref{lemma_local_ISS}]
Let $c_2=(1-\sigma)\delta_0$ in Lemma \ref{lemma_total_error_bounded_ISS}, and $\delta_1$ be equal to the $\delta_1^1$ associated with $c_2$. For any $i\in\mathbb{Z}_+$, if $\tilde{P}_i\in\mathcal{B}_{\delta_0}(P^*)$, then $[\hat{G}_i]_{uu}$ is invertible, $\hat{K}_{i+1}$ is stabilizing and  
\begin{align}
    \Vert \tilde{P}_{i+1} - P^*\Vert_F&\leq \Vert \mathcal{E}_i\Vert_F +\nonumber \left\Vert \mathcal{L}^{-1}_{\mathcal{A}(\mathscr{K}(\tilde{P}_i))}(S+\tilde{P}^T_iBR^{-1}B^T\tilde{P}_i) - P^* \right\Vert_F \nonumber \\
    &\leq \sigma\Vert \tilde{P}_i - P^*\Vert_F + c_3\Vert\Delta G_i\Vert_F \label{proof_lemma_local_ISS_inequality_1}\\
    &\leq  \sigma\Vert \tilde{P}_i - P^*\Vert_F + c_3\Vert\Delta G\Vert_\infty \label{proof_lemma_local_ISS_inequality_2}\\
    &<\sigma\delta_0 + c_3\delta_1<\sigma\delta_0 + c_2 =\delta_0, \label{proof_lemma_local_ISS_inequality_3}
\end{align}
where \eqref{proof_lemma_local_ISS_inequality_1} and \eqref{proof_lemma_local_ISS_inequality_3} are due to Lemmas \ref{lemma_local_exponentail_stable} and \ref{lemma_total_error_bounded_ISS}. By induction, \eqref{proof_lemma_local_ISS_inequality_1} to \eqref{proof_lemma_local_ISS_inequality_3} hold for all $i\in\mathbb{Z}_+$, thus by \eqref{proof_lemma_local_ISS_inequality_2},
\begin{equation*}
    \begin{split}
        \Vert \tilde{P}_{i} - P^*\Vert_F &\leq \sigma^2\Vert \tilde{P}_{i-2}-P^*\Vert_F + (\sigma + 1)c_3\Vert\Delta G\Vert_\infty\\
        &\leq \cdots \leq \sigma^{i}\Vert \tilde{P}_{0}-P^*\Vert_F + (1+ \cdots + \sigma^{i-1})c_3\Vert\Delta G\Vert_\infty\\
        &<\sigma^{i}\Vert \tilde{P}_{0}-P^*\Vert_F + \frac{c_3}{1-\sigma}\Vert\Delta G\Vert_\infty,
    \end{split}
\end{equation*}
which proves \ref{lemma_local_ISS_item_well_defined} and \ref{lemma_local_ISS_item_local_ISS} in Lemma \ref{lemma_local_ISS}. Then \eqref{K_close} implies \ref{lemma_local_ISS_item_K_bounded} in Lemma \ref{lemma_local_ISS}.

In terms of \ref{lemma_local_ISS_item_converging_input_converging_state} in Lemma \ref{lemma_local_ISS}, for any $\epsilon>0$, there exists a $i_1\in\mathbb{Z}_+$, such that $\sup\{\Vert\Delta G_i\Vert_F\}_{i=i_1}^\infty<\gamma^{-1}(\epsilon/2)$. Take $i_2\geq i_1$. For $i\geq i_2$, we have by \ref{lemma_local_ISS_item_local_ISS} in Lemma \ref{lemma_local_ISS},
\begin{equation*}
    \begin{split}
       \Vert \tilde{P}_{i}-P^*\Vert_F\leq\beta(\Vert\tilde{P}_{i_2}-P^*\Vert_F,i-i_2) + \epsilon/2 \leq \beta(c_7,i-i_2) + \epsilon/2.
    \end{split}
\end{equation*}
where the second inequality is due to the boundedness of $\tilde{P}_i$. Since $\lim_{i\rightarrow\infty}\beta(c_7,i-i_2)=0$, there is a $i_3\geq i_2$ such that $\beta(c_7,i-i_2)<\epsilon/2$ for all $i\geq i_3$, which completes the proof.
\end{proof}
\section{Appendix D: Proof of Theorem \ref{theorem_global}}\label{proof_theorem_global}
Notice that all the conclusions of Theorem \ref{theorem_global} can be implied by Lemma \ref{lemma_local_ISS} if
$$\delta_2<\min(\gamma^{-1}(\epsilon),\delta_1),\quad \tilde{P}_1\in\mathcal{B}_{\delta_0}(P^*)$$ 
for Procedure \ref{procedure_robust_policy_iteration}. Thus the proof of Theorem \ref{theorem_global} reduces to the proof of the following lemma.
\begin{lemma}\label{lemma_converge_to_neighbourhood}
Given a stabilizing $\hat{K}_1$, there exist $0<\delta_2<\min(\gamma^{-1}(\epsilon),\delta_1)$, $\bar{i}\in\mathbb{Z}_+$, $\Pi_2>0$ and $\kappa_2>0$, such that $[\hat{G}_i]_{uu}$ is invertible, $\hat{K}_i$ is stabilizing, $\Vert \tilde{P}_i\Vert_F<\Pi_2$, $\Vert \hat{K}_i\Vert_F<\kappa_2$, $i=1,\cdots,\bar{i}$, $\tilde{P}_{\bar{i}}\in\mathcal{B}_{\delta_0}(P^*)$, as long as $\Vert \Delta G\Vert_\infty<\delta_2$.
\end{lemma}
The next two lemmas, inspired by \cite[Lemma 5.1]{pmlr-v89-abbasi-yadkori19a}, state that under certain conditions on $\Vert \Delta G_i\Vert_F$, each element in $\{\hat{K}_i\}_{i=1}^{\bar{i}}$ is stabilizing, each element in $\{[\hat{G}_i]_{uu}\}_{i=1}^{\bar{i}}$ is invertible and $\{\tilde{P}_i\}_{i=1}^{\bar{i}}$ is bounded. For simplicity, in the following we assume $S>I_n$ and $R>I_m$. All the proofs still work for any $S>0$ and $R>0$, by suitable rescaling.
\begin{lemma}\label{lemma_stability}
If $\hat{K}_i$ is stabilizing, then $[\hat{G}_i]_{uu}$ is nonsingular and $\hat{K}_{i+1}$ is stabilizing, as long as $\Vert \Delta G_i\Vert_F < a_i$, where
\begin{equation*}
     a_i=\left(m(\sqrt{n} + \Vert \hat{K}_{i}\Vert_2)^2+m(\sqrt{n} + \Vert \hat{K}_{i+1}\Vert_2)^2\right)^{-1}.
\end{equation*}
Furthermore,
\begin{equation}
    \Vert\hat{K}_{i+1}\Vert_F \leq 2\Vert R^{-1}\Vert_F(1 + \Vert B^T\tilde{P}_iA\Vert_F).  \label{K_bound_by_P}
\end{equation}
\end{lemma}
\begin{proof}
By definition,
$$\Vert [\tilde{G}_i]_{uu}^{-1}([\hat{G}_i]_{uu}-[\tilde{G}_i]_{uu})\Vert_F < a_i\Vert [\tilde{G}_i]_{uu}^{-1}\Vert_F.$$
Since $R>I_m$, the eigenvalues $\lambda_j([\tilde{G}_i]_{uu}^{-1})\in(0,1]$ for all $1\leq j\leq m$. Then by the fact that for any $X\in\mathbb{S}^m$
$$\Vert X\Vert_F = \Vert \Lambda_X\Vert_F,\quad \Lambda_X = \mathrm{diag}\{\lambda_1(X),\cdots,\lambda_m(X)\},$$
we have 
\begin{equation}\label{R_inverse_pertrubed_bound}
    \Vert [\tilde{G}_i]_{uu}^{-1}([\hat{G}_i]_{uu}-[\tilde{G}_i]_{uu})\Vert_F< a_i\sqrt{m} < 0.5.
\end{equation}
Thus by \cite[Section 5.8]{horn2012matrix}, $[\hat{G}_i]_{uu}$ is invertible.

For any $x\in\mathbb{R}^{n}$ on the unit ball, define
$$\mathcal{X}_{\hat{K}_i} = \left[\begin{array}{c}
         I  \\
         -\hat{K}_i
    \end{array}\right]xx^T \left[\begin{array}{cc}
        I &  -\hat{K}_i^T
    \end{array}\right].$$
From \eqref{eRPI_PE} and \eqref{eRPI_PI} we have
$$x^T\mathcal{H}(\tilde{G}_i,\hat{K}_i)x = \trace(\tilde{G}_i\mathcal{X}_{\hat{K}_i}) = 0,$$
and
$$\trace(\hat{G}_i\mathcal{X}_{\hat{K}_{i+1}}) = \min_{K\in\mathbb{R}^{m\times n}} \trace(\hat{G}_i\mathcal{X}_K).$$
Then 
\begin{align}
    &\mathrm{tr}(\tilde{G}_i\mathcal{X}_{\hat{K}_{i+1}}) \leq \mathrm{tr}(\hat{G}_i\mathcal{X}_{\hat{K}_{i+1}}) + \Vert \Delta G_i\Vert_F \mathrm{tr}(\mathbf{1}\mathbf{1}^T\vert \mathcal{X}_{\hat{K}_{i+1}}\vert_{abs}) \nonumber\\   
    &\leq \mathrm{tr}(\hat{G}_i\mathcal{X}_{\hat{K}_{i}}) + \Vert \Delta G_i\Vert_F \mathbf{1}^T\vert \mathcal{X}_{\hat{K}_{i+1}}\vert_{abs}\mathbf{1}\nonumber \\
    &\leq \mathrm{tr}(\tilde{G}_i\mathcal{X}_{\hat{K}_{i}}) + \Vert \Delta G_i\Vert_F \mathbf{1}^T(\vert \mathcal{X}_{\hat{K}_{i}}\vert_{abs}+\vert \mathcal{X}_{\hat{K}_{i+1}}\vert_{abs})\mathbf{1} \nonumber\\
    &\leq \Vert \Delta G_i\Vert_F \mathbf{1}^T(\vert \mathcal{X}_{\hat{K}_{i}}\vert_{abs}+\vert \mathcal{X}_{\hat{K}_{i+1}}\vert_{abs})\mathbf{1}, \label{key_inequality_1}
\end{align}
where $\vert \mathcal{X}_{\hat{K}_{i}}\vert_{abs}$ denotes the matrix obtained from $\mathcal{X}_{\hat{K}_{i}}$ by taking the absolute value of each entry. Thus by (\ref{key_inequality_1}) and the definition of $\tilde{G}_i$, we have
\begin{equation}\label{key_inequality_2}
    x^T\mathcal{L}_{\mathcal{A}(\hat{K}_{i+1})}(\tilde{P}_i)x + \epsilon_1 \leq 0
\end{equation}
where 
\begin{equation*}
\begin{split}
    \epsilon_1 = x^T(S+\hat{K}_{i+1}^TR\hat{K}_{i+1})x - \Vert \Delta G_i\Vert_F \mathbf{1}^T(\vert\mathcal{X}_{\hat{K}_{i}}\vert_{abs}+\vert \mathcal{X}_{\hat{K}_{i+1}}\vert_{abs})\mathbf{1}.    
\end{split}
\end{equation*}
For any $x$ on the unit ball, $\vert\mathbf{1}^Tx\vert_{abs}\leq \sqrt{n}$. Similarly, for any $K\in\mathbb{R}^{m\times n}$, by the definition of induced matrix norm, $\vert\mathbf{1}^TKx\vert_{abs}\leq\Vert K\Vert_2 \sqrt{m}$. This implies
\begin{equation*}
    \left\vert\mathbf{1}^T\left[\begin{array}{c}
    I \\
    -K
    \end{array}\right]x\right\vert_{abs} = \left\vert\mathbf{1}^Tx - \mathbf{1}^TKx\right\vert_{abs} \leq \sqrt{m}(\sqrt{n} + \Vert K\Vert_2),
\end{equation*}
which means $\mathbf{1}^T\vert \mathcal{X}_K\vert_{abs}\mathbf{1}\leq m(\sqrt{n} + \Vert K\Vert_2)^2$. Thus
$$\Vert \Delta G_i\Vert_F \mathbf{1}^T(\vert \mathcal{X}_{\hat{K}_{i}}\vert_{abs}+\vert \mathcal{X}_{\hat{K}_{i+1}}\vert_{abs})\mathbf{1}<1.$$
Then $S>I_n$ leads to
$$x^T\left(\mathcal{A}(\hat{K}_{i+1})^T\tilde{P}_i\mathcal{A}(\hat{K}_{i+1})-\tilde{P}_i\right)x<0$$
for all $x$ on the unit ball. So $\hat{K}_{i+1}$ is stabilizing by the Lyapunov criterion \cite[Lemma $12.2'$]{wonham1974linear}.

By definition, 
\begin{align}
        \Vert\hat{K}_{i+1}\Vert_F
        &\leq \Vert [\hat{G}_i]_{uu}^{-1}\Vert_F(1 + \Vert B^T\tilde{P}_iA\Vert_F)\nonumber \\
        &\leq \Vert [\tilde{G}_i]_{uu}^{-1}\Vert_F(1-\Vert [\tilde{G}_i]_{uu}^{-1}([\hat{G}_i]_{uu}-[\tilde{G}_i]_{uu})\Vert_F)^{-1}\nonumber(1 + \Vert B^T\tilde{P}_iA\Vert_F) \nonumber\\
        &\leq 2\Vert R^{-1}\Vert_F(1 + \Vert B^T\tilde{P}_iA\Vert_F).
\end{align}
where the second inequality comes from \cite[Inequality (5.8.2)]{horn2012matrix}, and the last inequality is due to \eqref{R_inverse_pertrubed_bound}. This completes the proof.
\end{proof}
\begin{lemma}\label{lemma_boundedness}
For any $\bar{i}\in\mathbb{Z}_+$, $\bar{i}>0$, if
\begin{equation}\label{error_condition_bounded}
   \Vert \Delta G_i\Vert_F< (1+i^2)^{-1}a_i,\quad i=1,\cdots, \bar{i},
\end{equation}
where $a_i$ is defined in Lemma \ref{lemma_stability}, then
\begin{equation*}
    \Vert\tilde{P}_i\Vert_F\leq 6\Vert \tilde{P}_1\Vert_F,\quad \Vert\hat{K}_{i}\Vert_F\leq C_0,
\end{equation*}
for $i=1,\cdots,\bar{i}$, where
$$ C_0 = \max\left\{ \Vert\hat{K}_1 \Vert_F, 2\Vert R^{-1}\Vert_F\left(1+6\Vert B^T\Vert_F\Vert \tilde{P}_1\Vert_F\Vert A\Vert_F\right)\right\}.$$
\end{lemma}
\begin{proof}
Inequality (\ref{key_inequality_2}) yields
\begin{equation}\label{key_inequality_3}
    \mathcal{L}_{\mathcal{A}(\hat{K}_{i+1})}(\tilde{P}_i) + (S+\hat{K}_{i+1}^TR\hat{K}_{i+1}) - \epsilon_{2,i}I < 0.
\end{equation}
where
$$\epsilon_{2,i} = \Vert \Delta G_i\Vert_F \mathbf{1}^T(\vert \mathcal{X}_{\hat{K}_{i}}\vert_{abs}+\vert \mathcal{X}_{\hat{K}_{i+1}}\vert_{abs})\mathbf{1}<1.$$
Inserting (\ref{eRPI_PE}) into above inequality, and using \ref{Lyapunov_monotone} in Lemma \ref{Lyapunov_properties}, we have
\begin{equation}\label{key_inequality_4_1}
    \tilde{P}_{i+1} < \tilde{P}_{i} + \epsilon_{2,i}\mathcal{L}_{\mathcal{A}(\hat{K}_{i+1})}^{-1}(-I).
\end{equation}
With $S>I_n$, (\ref{key_inequality_3}) yields
\begin{equation*}
    \mathcal{L}_{\mathcal{A}(\hat{K}_{i+1})}(\tilde{P}_i) + (1 - \epsilon_{2,i})I < 0.
\end{equation*}
Similar to (\ref{key_inequality_4_1}), we have
\begin{equation}\label{key_inequality_4_2}
    \mathcal{L}_{\mathcal{A}(\hat{K}_{i+1})}^{-1}(-I)<\frac{1}{1-\epsilon_{2,i}}\tilde{P}_i.
\end{equation}
From (\ref{key_inequality_4_1}) and (\ref{key_inequality_4_2}), we obtain
\begin{equation*}
    \tilde{P}_{i+1}<\left(1+\frac{\epsilon_{2,i}}{1-\epsilon_{2,i}}\right)\tilde{P}_i.
\end{equation*}
By definition of $\epsilon_{2,i}$ and condition \eqref{error_condition_bounded},
$$\frac{\epsilon_{2,i}}{1-\epsilon_{2,i}} \leq \frac{1}{i^2},\quad i=1,\cdots,\bar{i}.$$
Then \cite[\S 28. Theorem 3]{InfinitSeries} yields
$$\tilde{P}_i\leq 6\tilde{P}_1,\quad i=1,\cdots,\bar{i}.$$
An application of \eqref{K_bound_by_P} completes the proof.
\end{proof}
Now we are ready to prove Lemma \ref{lemma_converge_to_neighbourhood}.
\begin{proof}[Proof of Lemma \ref{lemma_converge_to_neighbourhood}]
Consider Procedure \ref{procedure_robust_policy_iteration} confined to the first $\bar{i}$ iterations, where $\bar{i}$ is a sufficiently large integer to be determined later in this proof. Suppose
\begin{equation}\label{error_bound_2}
        \Vert\Delta G_i\Vert_F<b_{\bar{i}}\triangleq\frac{1}{2m(1+\bar{i}^2)}\left(\sqrt{n} + C_0\right)^{-2}.
\end{equation}
Condition (\ref{error_bound_2}) implies condition (\ref{error_condition_bounded}). Thus $\hat{K}_i$ is stabilizing, $[\hat{G}_i]_{uu}^{-1}$ is invertible, $\Vert\tilde{P}_i\Vert_F$ and $\Vert \hat{K}_i\Vert_F$ are bounded. By (\ref{eRPI_PE}) we have
\begin{equation*}
\begin{split}
\mathcal{L}_{\mathcal{A}(\hat{K}_{i+1})}(\tilde{P}_{i+1}-\tilde{P}_{i})&= -S-\hat{K}_{i+1}^TR\hat{K}_{i+1}-\mathcal{L}_{\mathcal{A}(\hat{K}_{i+1})}(\tilde{P}_{i}).
\end{split}
\end{equation*}
Letting $E_i = \hat{K}_{i+1} - \mathscr{K}(\tilde{P}_i)$, the above equation can be rewritten as
\begin{equation}\label{perturbed_newton}
    \tilde{P}_{i+1} = \tilde{P}_i - \mathcal{N}(\tilde{P_i}) + \mathcal{L}_{\mathcal{A}(\mathscr{K}(\tilde{P_i}))}^{-1}(\mathscr{E}_i),
\end{equation}
where $\mathcal{N}(\tilde{P_i}) = \mathcal{L}_{\mathcal{A}(\mathscr{K}(\tilde{P_i}))}^{-1}\circ\mathcal{R}(\tilde{P}_i),$ and 
\begin{equation*}
    \begin{split}
    \mathcal{R}(Y) &= A^TYA-Y-A^TYB(R+B^TYB)^{-1}B^TYA+S,  \\
    \mathscr{E}_i &= - E_i^T\mathscr{R}(\tilde{P}_{i+1})E_i + E_i^T\mathscr{R}(\tilde{P}_{i+1})\left(\mathscr{K}(\tilde{P}_{i+1})-\mathscr{K}(\tilde{P}_i)\right)  \\
    &+\left(\mathscr{K}(\tilde{P}_{i+1})-\mathscr{K}(\tilde{P}_i)\right)^T\mathscr{R}(\tilde{P}_{i+1})E_i.
    \end{split}
\end{equation*}
Given $\hat{K}_1$, let $\mathcal{M}_{\bar{i}}$ denote the set of all possible $\tilde{P}_i$, generated by (\ref{perturbed_newton}) under condition (\ref{error_bound_2}). By definition, $\{\mathcal{M}_j\}_{j=1}^\infty$ is a nondecreasing sequence of sets, i.e., $\mathcal{M}_1\subset \mathcal{M}_2 \subset \cdots$. Define $\mathcal{M} = \cup_{j=1}^\infty\mathcal{M}_j$, $\mathcal{D} = \{P\in \mathbb{S}^n\ \vert\ \Vert P\Vert_F\leq 6\Vert \tilde{P}_1\Vert_F\}$. Then by Lemma \ref{lemma_boundedness} and Theorem \ref{theorem_standard_PI}, $\mathcal{M}\subset\mathcal{D}$; $\mathcal{M}$ is compact; $\mathcal{A}(\mathscr{K}(P))$ is stable for any $P\in\mathcal{M}$.\par
Now we prove that $\mathcal{N}(\cdot)$ is Lipschitz continuous on $\mathcal{M}$. Using \eqref{equivalent_operator_matrix}, for any $P^1,P^2\in\mathcal{M}$ we have
\begin{align}\label{Newton_Lipschitz}
        \Vert\mathcal{N}(P^1)-\mathcal{N}(P^2)\Vert_F&= \Vert\mathscr{A}^{-1}(\mathcal{A}(\mathscr{K}(P^1)))\vect(\mathcal{R}(P^1)) -\mathscr{A}^{-1}(\mathcal{A}(\mathscr{K}(P^2)))\vect(\mathcal{R}(P^2))\Vert_2 \nonumber\\
        %&\leq\Vert\mathscr{A}^{-1}(\mathcal{A}(\mathscr{K}(P^1)))\vect(\mathcal{R}(P^1)) -\mathscr{A}^{-1}(\mathcal{A}(\mathscr{K}(P^1)))\vect(\mathcal{R}(P^2))\Vert_2 +\nonumber\\
        %&\Vert\mathscr{A}^{-1}(\mathcal{A}(\mathscr{K}(P^1)))\vect(\mathcal{R}(P^2)) -\mathscr{A}^{-1}(\mathcal{A}(\mathscr{K}(P^2)))\vect(\mathcal{R}(P^2))\Vert_2\nonumber\\
        &\leq\Vert\mathscr{A}^{-1}(\mathcal{A}(\mathscr{K}(P^1)))\Vert_2\Vert\mathcal{R}(P^1) -\mathcal{R}(P^2)\Vert_F +\nonumber\\
        &\Vert\mathcal{R}(P^2)\Vert_F\Vert\mathscr{A}^{-1}(\mathcal{A}(\mathscr{K}(P^1))) -\mathscr{A}^{-1}(\mathcal{A}(\mathscr{K}(P^2)))\Vert_2\nonumber\\
        &\leq L \Vert P^1 -P^2 \Vert_F
\end{align}
with some Lipschitz constant $L>0$, where the last inequality is due to the fact that matrix inversion, $\mathcal{A}(\cdot)$, $\mathscr{K}(\cdot)$ and $\mathcal{R}(\cdot)$ are locally Lipschitz, thus Lipschitz on compact set $\mathcal{M}$.

Define $\{P_{k\vert i}\}_{k=0}^{\infty}$ as the sequence generated by \eqref{Kleinman_one_step} with $P_{0\vert i}=\tilde{P}_i$.
Similar to \eqref{perturbed_newton}, we have
\begin{equation}\label{exact_newton}
    P_{k+1\vert i} = P_{k\vert i} - \mathcal{N}(P_{k\vert i}), \quad k\in\mathbb{Z}_+.
\end{equation}
By Theorem \ref{theorem_standard_PI} and the fact that $\mathcal{M}$ is compact, there exists $k_0\in\mathbb{Z}_+$, such that 
\begin{equation}\label{triangle_inequality_first_part}
\Vert P_{k_0\vert i}-P^*\Vert_F<\delta_0/2, \qquad \forall P_{0\vert i}\in\mathcal{M}.
\end{equation}
Suppose
\begin{equation}\label{perturb_small}
    \Vert\mathcal{L}_{\mathcal{A}(\mathscr{K}(\tilde{P}_{i+j}))}^{-1}(\mathscr{E}_{i+j})\Vert_F<\mu,\qquad j=0,\cdots, \bar{i} - i.
\end{equation}
We find an upper bound on $\Vert P_{k\vert i}-\tilde{P}_{i+k}\Vert_F$. Notice that from (\ref{perturbed_newton}) and (\ref{exact_newton}),
\begin{equation*}
    \begin{split}
        P_{k\vert i} = P_{0\vert i} - \sum_{j=0}^{k-1} \mathcal{N}(P_{j\vert i}), \qquad \tilde{P}_{i+k} = \tilde{P}_i - \sum_{j=0}^{k-1} \mathcal{N}(\tilde{P}_{i+j}) + \sum_{j=0}^{k-1} \mathcal{L}_{\mathcal{A}(\mathscr{K}(\tilde{P}_{i+j}))}^{-1}(\mathscr{E}_{i+j}).
    \end{split}
\end{equation*}
Then (\ref{Newton_Lipschitz}) and (\ref{perturb_small}) yield 
\begin{equation*}
    \begin{split}
        \Vert P_{k\vert i} - \tilde{P}_{i+k}\Vert_F \leq k\mu + \sum_{j=0}^{k-1} L \Vert P_{j\vert i} - \tilde{P}_{i+j}\Vert_F.
    \end{split}
\end{equation*}
An application of the Gronwall inequality \cite[Theorem 4.1.1.]{agarwal2000difference} to the above inequality implies
\begin{equation}\label{continuous_dependence}
    \Vert P_{k\vert i} - \tilde{P}_{i+k}\Vert_F \leq k\mu + L\mu \sum_{j=0}^{k-1}j (1+L)^{k-j-1}.
\end{equation}
By \ref{Lyapunov_solution} in Lemma \ref{Lyapunov_properties}, 
$$\mathcal{L}^{-1}_{\mathcal{A}(\mathscr{K}(P^1))}(-I) = \sum_{k=0}^\infty (\mathcal{A}^T(\mathscr{K}(P^1)))^k (\mathcal{A}(\mathscr{K}(P^1)))^k,\quad P^1\in\mathcal{M}.$$
By definition of the set $\mathcal{M}$, there exist $G^1$ and stabilizing control gain $K^1$ associated with $P^1\in\mathcal{M}$, such that $\mathcal{H}(G^1,K^1)=0$. So for any $x\in\mathbb{R}^n$,
\begin{equation*}
    x^T\mathcal{R}(P^1)x = \min_{K\in\mathbb{R}^{m\times n}} x^T\mathcal{H}(G^1,K)x \leq 0.
\end{equation*}
This implies
$$\mathcal{L}_{\mathcal{A}(\mathscr{K}(P^1))}(P^1)\leq -S - \mathscr{K}^T(P^1)R\mathscr{K}(P^1)<-I.$$
An application of \ref{Lyapunov_monotone} in Lemma \ref{Lyapunov_properties} to the above inequality leads to
\begin{equation}\label{bound_on_H}
    \Vert \mathcal{L}^{-1}_{\mathcal{A}(\mathscr{K}(P^1))}(-I)\Vert_F<P^1 \leq 6\Vert \tilde{P}_1\Vert_F, \qquad \forall P^1\in\mathcal{M}\subset\mathcal{D}.
\end{equation}
Then the error term in (\ref{perturbed_newton}) satisfies
\begin{align}\label{perturbed_inverse_lypunov}
    \left\Vert \mathcal{L}_{\mathcal{A}(\mathscr{K}(\tilde{P_i}))}^{-1}(\mathscr{E}_i)\right\Vert_F 
    = \left\Vert \sum_{k=0}^\infty (\mathcal{A}^T(\mathscr{K}(\tilde{P_i})))^k\otimes (\mathcal{A}(\mathscr{K}(\tilde{P_i})))^k \vect\left( \mathscr{E}_i\right)\right\Vert_2 \leq C_1\Vert \mathscr{E}_i\Vert_F,
\end{align}
where $C_1$ is a constant and the inequality is due to (\ref{bound_on_H}).

Let $\bar{i}>k_0$, and $k=k_0$, $i = \bar{i}-k_0$ in \eqref{continuous_dependence}. Then by condition \eqref{error_bound_2}, Lemma \ref{lemma_boundedness}, (\ref{perturb_small}), (\ref{continuous_dependence}), and \eqref{perturbed_inverse_lypunov}, there exists $i_0\in\mathbb{Z}_+$, $i_0>k_0$, such that $\Vert P_{k_0\vert \bar{i}-k_0} -\tilde{P}_{\bar{i}}\Vert_F<\delta_0/2$, for all $\bar{i}\geq i_0$. Setting $i = \bar{i}-k_0$ in \eqref{triangle_inequality_first_part}, the triangle inequality yields $\tilde{P}_{\bar{i}}\in\mathcal{B}_{\delta_0}(P^*)$, for $\bar{i}\geq i_0$. Then in \eqref{error_bound_2}, choosing $\bar{i}\geq i_0$ such that $\delta_2 = b_{\bar{i}}<\min(\gamma^{-1}(\epsilon),\delta_1)$ completes the proof.
\end{proof}

\section{Appendix E: Proof of Theorem \ref{theorem_stochastic_algorithm_convergence}}\label{proof_algorithm_convergence}
For given $\hat{K}_1$, let $\mathcal{F}$ denote the set of control gains (including $\hat{K}_1$) generated by Procedure \ref{procedure_robust_policy_iteration} with all possible $\{\Delta G_i\}_{i=1}^\infty$ satisfying $\Vert \Delta G\Vert_\infty<\delta_2$, where $\delta_2$ is the one in Theorem \ref{theorem_global}. The following result is firstly derived. 
\begin{lemma}\label{Lemma_policy_evaluation_error_bounded}
Under Assumption \ref{assumption_PE}, there exist $T_0\in\mathbb{Z}_+$ and $M_0\in\mathbb{Z}_+$, such that for any $T\geq T_0$ and any $M\geq M_0$, $\hat{K}_i\in\mathcal{F}$ implies $\Vert\Delta G_i\Vert_F< \delta_2$, almost surely.
\end{lemma}
\begin{proof}
The task is to show that each term in \eqref{error_decomposition} is less than $\delta_2/3$ almost surely.

We firstly study term $\Vert \hat{P}_{i,T}-\tilde{P}_i\Vert_F$. 
Define $\hat{p}_{i,j} = \vect(\hat{P}_{i,j})$, then by Lemma \ref{lemma_relationship_vec_svec}, lines from \ref{algorithm_policy_evaluation_begin} to \ref{algorithm_policy_teration_end} in Algorithm \ref{algorithm_off_policy_stochastic} can be rewritten as
\begin{equation}\label{policy_evaluation_approximate_vector}
\begin{split}
    \hat{p}_{i,j+1} &= \vect(\mathcal{H}(\mathcal{H}(\svec^{-1}(\Phi_M^\dagger\Psi_M\svec(\hat{P}_{i,j}) + \Phi_M^\dagger\Xi_M),0),\hat{K}_i))  \\
    & = \mathcal{T}^1(\Phi_M,\Psi_M,\hat{K}_i)\hat{p}_{i,j} + \mathcal{T}^2(\Phi_M,\Xi_M,\hat{K}_i),\quad \hat{p}_{i,0}=0,
    % & = \left(\left[I_n,-\hat{K}^T_i\right]\otimes\left[I_n,-\hat{K}^T_i\right]\right)\left(\left[I_{m+n},0\right]\otimes\left[I_{m+n},0\right]\right)\vect(\svec^{-1}(\Phi_M^\dagger\Psi_M\svec(\hat{P}_{i,j}) + \Phi_M^\dagger\Xi_M))   \\
    % & = \left(\left[I_n,-\hat{K}^T_i\right]\otimes\left[I_n,-\hat{K}^T_i\right]\right)\left(\left[I_{m+n},0\right]\otimes\left[I_{m+n},0\right]\right)D_{m+n+1}(\Phi_M^\dagger\Psi_MD_n^\dagger\hat{p}_{i,j} + \Phi_M^\dagger\Xi_M)
\end{split}
\end{equation}
where $Y = \svec^{-1}(\svec(Y))$, for any $Y\in\mathbb{S}^n$, and
\begin{equation*}
\begin{split}
\mathcal{T}^1(\Phi_M,\Psi_M,\hat{K}_i) &= \left(\left[I_n,-\hat{K}^T_i\right]\otimes\left[I_n,-\hat{K}^T_i\right]\right)\left(\left[I_{m+n},0\right]\otimes\left[I_{m+n},0\right]\right)D_{m+n+1}\Phi_M^\dagger\Psi_MD_n^\dagger,    \\
\mathcal{T}^2(\Phi_M,\Xi_M,\hat{K}_i) &= \left(\left[I_n,-\hat{K}^T_i\right]\otimes\left[I_n,-\hat{K}^T_i\right]\right)\left(\left[I_{m+n},0\right]\otimes\left[I_{m+n},0\right]\right)D_{m+n+1}\Phi_M^\dagger\Xi_M.
\end{split}
\end{equation*}
Set $K=\hat{K}_i$ and insert \eqref{exact_param_estimation} into \eqref{policy_evaluation_model_based}. Similar derivations yield
\begin{equation}\label{policy_evaluation_exact_estimation_vector}
    \tilde{p}_{i,j+1} =
    \mathcal{T}^1(\varphi_1,\varphi_2,\hat{K}_i)\tilde{p}_{i,j} + \mathcal{T}^2(\varphi_1,\varphi_3,\hat{K}_i),\quad \tilde{p}_{i,0}=0.
\end{equation}
By Assumption \ref{assumption_PE}, \eqref{policy_evaluation_exact_estimation_vector} is identical to
\begin{equation}\label{policy_evaluation_exact_estimation}
    \tilde{p}_{i,j+1} = (A-B\hat{K}_i)^T\otimes(A-B\hat{K}_i)^T\tilde{p}_{i,j} + \vect(S + \hat{K}_i^TR\hat{K}_i),\quad \tilde{p}_{i,0}=0,
\end{equation}
with
\begin{equation}\label{data_model_identical}
    \mathcal{T}^1(\varphi_1,\varphi_2,\hat{K}_i) = (A-B\hat{K}_i)^T\otimes(A-B\hat{K}_i)^T,\quad \mathcal{T}^2(\varphi_1,\varphi_3,\hat{K}_i)= \vect(S + \hat{K}_i^TR\hat{K}_i).
\end{equation}
Since $\hat{K}_i\in\mathcal{F}$ is stabilizing,
\begin{equation}\label{policy_evaluation_converge}
    \lim_{j\rightarrow\infty} \tilde{P}_{i,j} = \tilde{P}_i
\end{equation} 
where $\tilde{p}_{i,j} = \vect(\tilde{P}_{i,j})$ and $\tilde{P}_i$ is the unique solution of \eqref{algebraic_Lyapunov_equation} with $K=\hat{K}_i$. By definition and Theorem \ref{theorem_global}, $\bar{\mathcal{F}}$ is bounded, thus compact. Let $\mathcal{V}$ be the set of unique solutions of \eqref{algebraic_Lyapunov_equation} with $K\in\mathcal{F}$. Then by Theorem \ref{theorem_global} $\mathcal{V}$ is bounded. So any control gain in $\bar{\mathcal{F}}$ is stabilizing, otherwise it contradicts the boundedness of $\mathcal{V}$. Define
$$\mathcal{F}_1 = \{(A-BK)^T\otimes(A-BK)^T\vert K\in\bar{\mathcal{F}}\}.$$
By continuity, $\mathcal{F}_1$ is a compact set of stable matrices, and there exists a $\delta_3>0$, such that any $X\in\bar{\mathcal{F}}_2$ is stable, where
$$\mathcal{F}_2 = \{X\vert X\in\mathcal{B}_{\delta_3}(Y),Y\in\mathcal{F}_1\}.$$
% Subtracting \eqref{policy_evaluation_approximate_vector} from \eqref{policy_evaluation_exact_estimation_vector} gives
% \begin{equation}\label{error_difference_equation}
%     \Delta p_{i,j+1} = \Delta \mathcal{T}^1_{M,i}\Delta p_{i,j} + \Delta \mathcal{T}^2_{M,i},\quad \Delta p_{i,0} = 0,
% \end{equation}
% where $\Delta p_{i,j} = \tilde{p}_{i,j} - \hat{p}_{i,j}$ and
Define
$$\Delta\mathcal{T}^1_{M,i} = \mathcal{T}^1(\varphi_1,\varphi_2,\hat{K}_i) - \mathcal{T}^1(\Phi_M,\Psi_M,\hat{K}_i), \quad \Delta\mathcal{T}^2_{M,i} = \mathcal{T}^2(\varphi_1,\varphi_3,\hat{K}_i) - \mathcal{T}^2(\Phi_M,\Xi_M,\hat{K}_i).$$
The boundedness of $\mathcal{F}$, \eqref{ergodic_converge} and \eqref{data_model_identical} implies the existence of $M_1>0$, such that for any $M\geq M_1$, any $\hat{K}_i\in\mathcal{F}$, almost surely
\begin{equation}\label{belong_to_compact_stable_set}
    \mathcal{T}^1(\Phi_M,\Psi_M,\hat{K}_i)\in\bar{\mathcal{F}}_2, \quad \mathcal{T}^2(\Phi_M,\Xi_M,\hat{K}_i)< C_2,
\end{equation}
where $C_2>0$ is a constant.
Then \eqref{policy_evaluation_approximate_vector} admits a unique stable equilibrium, that is,
\begin{equation}\label{data_policy_evaluation_converge}
    \lim_{j\rightarrow\infty}\hat{P}_{i,j} = \hat{P}_i
\end{equation}
for some $\hat{P}_i\in\mathbb{S}^n$, and from \eqref{policy_evaluation_approximate_vector}, \eqref{policy_evaluation_exact_estimation_vector}, \eqref{policy_evaluation_converge} and \eqref{data_policy_evaluation_converge}, we have
\begin{equation*}
    \begin{split}
        \tilde{p}_i &= \vect(\tilde{P}_i) = \left( I_{n^2} - \mathcal{T}^1(\varphi_1,\varphi_2,\hat{K}_i) \right)^{-1}\mathcal{T}^2(\varphi_1,\varphi_3,\hat{K}_i), \\
        \hat{p}_i &= \vect(\hat{P}_i) = \left( I_{n^2} - \mathcal{T}^1(\Phi_M,\Psi_M,\hat{K}_i) \right)^{-1}\mathcal{T}^2(\Phi_M,\Xi_M,\hat{K}_i).
    \end{split}
\end{equation*}
Thus by \eqref{linear_perturbation}, for any $M\geq M_1$, any $\hat{K}_i\in\mathcal{F}$, almost surely
\begin{align*}
\Vert \hat{P}_i - \tilde{P}_i\Vert_F&\leq \left\Vert\left( I_{n^2} - \mathcal{T}^1(\varphi_1,\varphi_2,\hat{K}_i)\right)^{-1}\right\Vert_F\left(\Vert \Delta\mathcal{T}^2_{M,i} \Vert_2 +\right.\\
&\left.\left\Vert\left( I_{n^2} - \mathcal{T}^1(\Phi_M,\Psi_M,\hat{K}_i) \right)^{-1}\right\Vert_F\left\Vert\mathcal{T}^2(\Phi_M,\Xi_M,\hat{K}_i)\right\Vert_2\Vert\Delta\mathcal{T}^1_{M,i} \Vert_F\right) \\
&\leq c_8 \Vert \Delta\mathcal{T}^2_{M,i} \Vert_F + c_9\Vert\Delta\mathcal{T}^1_{M,i} \Vert_F
\end{align*}
where $c_8$ and $c_9$ are some positive constants, and the last inequality is due to \eqref{data_model_identical}, \eqref{belong_to_compact_stable_set} and the fact that $\mathcal{F}_1$ and $\bar{\mathcal{F}}_2$ are compact sets of stable matrices. Then for any $\epsilon_1>0$, the boundedness of $\mathcal{F}$ and \eqref{ergodic_converge} implies the existence of $M_2\geq M_1$, such that for any $M\geq M_2$, almost surely
\begin{equation}\label{PE_error_first_half}
    \Vert \hat{P}_i - \tilde{P}_i\Vert_F<\epsilon_1/2,
\end{equation}
as long as $\hat{K}_i\in\mathcal{F}$. By Lemma \ref{uniform_exp_stab_over_compact_set} and \eqref{belong_to_compact_stable_set}, for any $M\geq M_2$ and any $\hat{K}_i\in\mathcal{F}$,
$$\Vert \hat{p}_{i,j+1}-\hat{p}_i\Vert_2\leq a_0b_0^k\Vert \hat{p}_i\Vert_2\leq a_1b_0^k,$$
for some $a_0>0$, $1>b_0>0$ and $a_1>0$. Therefore there exists a $T_1>0$, such that for any $T\geq T_1$, and any $M\geq M_2$, almost surely
\begin{equation}\label{PE_error_second_half}
    \Vert \hat{P}_{i,T}-\hat{P}_i\Vert_F<\epsilon_1/2,
\end{equation}
as long as $\hat{K}_i\in\mathcal{F}$. Synthesizing \eqref{PE_error_first_half} and \eqref{PE_error_second_half} yields
\begin{equation}\label{triangle_inequality_3}
    \Vert \hat{P}_{i,T} - \tilde{P}_i\Vert_F<\epsilon_1,
\end{equation}
almost surely for any $T\geq T_1$, any $M\geq M_2$, as long as $\hat{K}_i\in\mathcal{F}$. Since $\epsilon_1$ is arbitrary, we can choose $\epsilon_1$ such that almost surely
$$\Vert \hat{P}_{i,T}-\tilde{P}_i\Vert_F<\epsilon_1\leq\frac{\delta_2}{3}.$$

Secondly, for term $\Vert Q(\hat{P}_{i,T}) - Q(\tilde{P}_i)\Vert_F$, by \eqref{triangle_inequality_3} there exist $T_2\geq T_1$, $M_3\geq M_2$, such that almost surely
\begin{equation*}
    \Vert Q(\hat{P}_{i,T}) - Q(\tilde{P}_i)\Vert_F<\frac{\delta_2}{3}
\end{equation*}
for any $T\geq T_2$, any $M\geq M_3$, as long as $\hat{K}_i\in\mathcal{F}$.

Finally, since $\mathcal{V}$ is bounded, by \eqref{triangle_inequality_3} $\hat{P}_{i,T}$ is also almost surely bounded. Thus from lines \ref{algorithm_estimate_final_Q_1} to \ref{algorithm_estimate_final_Q_2} in Algorithm \ref{algorithm_off_policy_stochastic}, and \eqref{ergodic_converge}, there exists $M_4\geq M_3$, such that
$$\Vert \hat{Q}_{i,T} - Q(\hat{P}_{i,T})\Vert_F<\frac{\delta_2}{3}$$
for any $M\geq M_4$ and any $T\geq T_2$, as long as $\hat{K}_i\in\mathcal{F}$.

Setting $T_0 = T_2$ and $M_0 = M_4$ yields $\Vert\Delta G_i\Vert<\delta_2$.
\end{proof}
Now we are ready to prove Theorem \ref{theorem_stochastic_algorithm_convergence}.
\begin{proof}
Since $\hat{K}_1\in\mathcal{F}$, Lemma \ref{Lemma_policy_evaluation_error_bounded} implies $\Vert\Delta G_1\Vert_F<\delta_2$ almost surely. By definition, $\hat{K}_2\in\mathcal{F}$. Thus $\Vert\Delta G_i\Vert_F<\delta_2, i=1,2,\cdots$ almost surely by mathematical induction. Then Theorem \ref{theorem_global} completes the proof.
\end{proof}

\section{Appendix F: Experiments (Cont'd)}\label{Appendix_experiments}
In the experiments of Figure \ref{exp_noise_mag}, initial control gain is $\hat{K}_1 = 0_{2\times 3}$, number of policy iterations $N=5$, number of iterations for policy evaluation $T=45$. The three features we use to evaluate the performances of the algorithms (fraction stable, relative error (Avg.) and Relative Error (Var.)) in Figure \ref{exp_noise_mag} are computed in the same way to those in Figure \ref{exp_results}.
\begin{figure}[!htb]
    \centering
    \includegraphics[width=\linewidth]{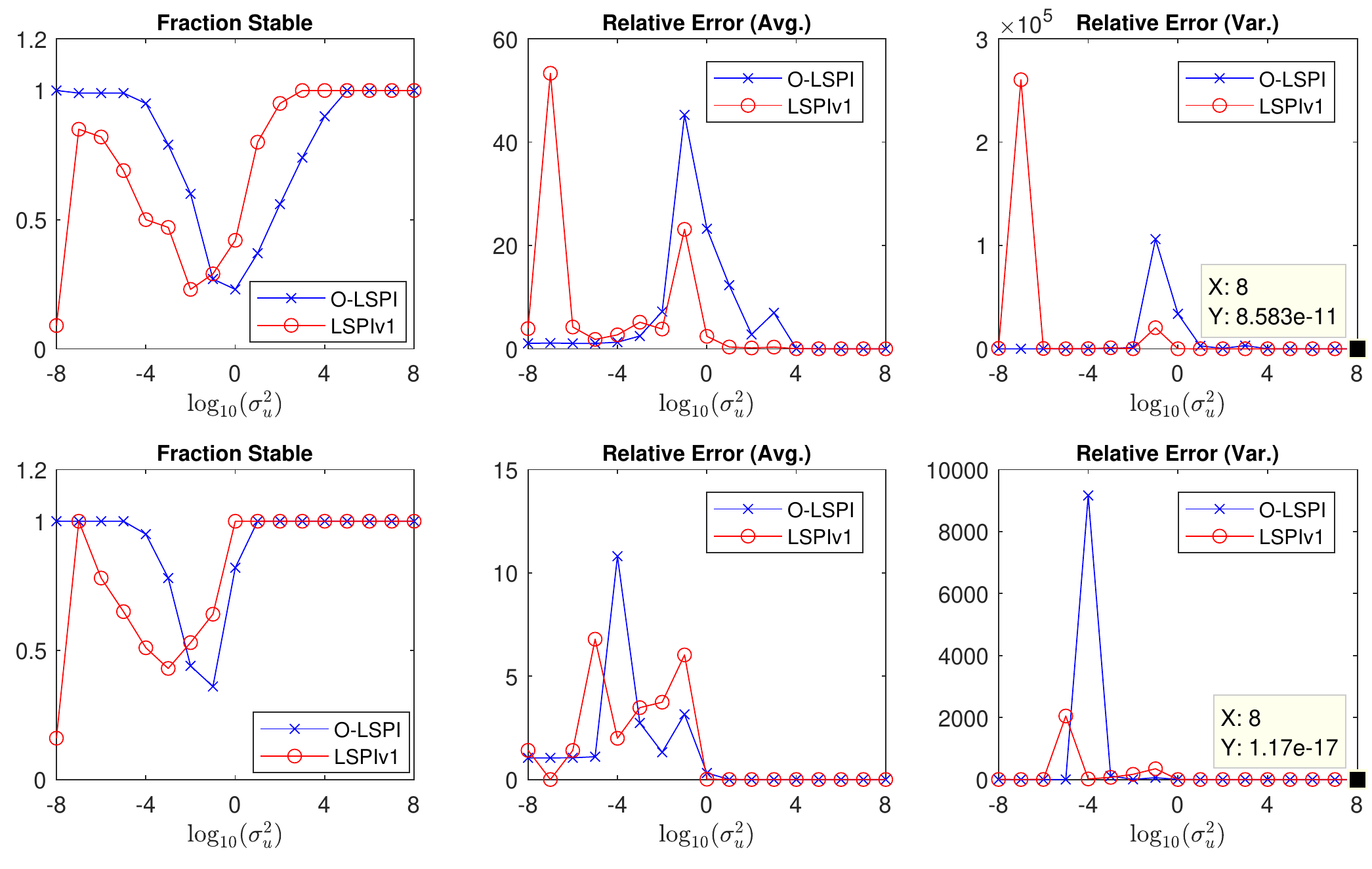}
    \caption{The performances of O-LSPI and LSPIv1 on the dynamics of \cite{krauth2019finite} with different values of exploration variances $\sigma^2_u$. In the upper row $M=10^2$ while in the lower row $M=10^4$.}\label{exp_noise_mag}
\end{figure}

\end{document}